\newif\ifMS
\MSfalse

\ifMS
\documentclass[journal,twocolumn]{IEEEtran}
\else
\documentclass[a4paper]{article}
\usepackage[margin=1in]{geometry}
\fi

\usepackage{mathtools}
\usepackage{amsmath,amssymb,amsfonts}%
\usepackage{amsthm}%
\usepackage{mathrsfs}%
\usepackage{textcomp}%
\usepackage{manyfoot}%
\usepackage{xparse}
\usepackage[normalem]{ulem}
\usepackage{hyperref}
\hypersetup{
	colorlinks,
	linkcolor=blue,
	citecolor=blue,
	urlcolor=blue,
}
\usepackage[capitalize]{cleveref}

\DeclarePairedDelimiterXPP{\opnorm}[1]{}{\lVert}{\rVert}{_{\mathrm{op}}}{#1}
\DeclarePairedDelimiterXPP{\nucnorm}[1]{}{\lVert}{\rVert}{_{*}}{#1}
\DeclarePairedDelimiterXPP{\normlt}[1]{}{\lVert}{\rVert}{_{L_2}}{#1}
\DeclarePairedDelimiterXPP{\nucnormlt}[1]{}{\lVert}{\rVert}{_{*,L_2}}{#1}
\DeclarePairedDelimiterXPP{\opnormlt}[1]{}{\lVert}{\rVert}{_{\mathrm{op},L_2}}{#1}
\DeclarePairedDelimiterXPP{\normFlt}[1]{}{\lVert}{\rVert}{_{\mathrm{F},L_2}}{#1}
\DeclarePairedDelimiterXPP{\iplt}[2]{}{\langle}{\rangle}{_{L_2}}{#1,#2}

\newcommand{\st}{\text{ s.t. }}

\newcommand{\tansp}{T_U}
\newcommand{\Ptansp}{\scrP_{\tansp}}

\newcommand{\PG}{P_{G}}
\newcommand{\PGp}{P_{G^\perp}}

\newcommand{\real}{\operatorname{Re}}
\newcommand{\imag}{\operatorname{Im}}
\newcommand{\gtrank}{r}

\newcommand{\SFp}{\F^p_1}

\newcommand{\crdvar}{w}
\newcommand{\incohp}{\mu}

\newcommand{\recmap}{T_{y,\lambda}}
\newcommand{\constF}{c_{\F}}
\newcommand{\constR}{c_{\R}}
\newcommand{\constC}{c_{\C}}

\usepackage{import}
\usepackage{amsthm}
\usepackage{xparse}
\usepackage{import}
\usepackage[normalem]{ulem}

\makeatletter
\@ifpackageloaded{unicode-math}{%
	\newcommand{\mymathbold}{\symbf}%
}{%
	\usepackage{bm}%
	\newcommand{\mymathbold}{\bm}%
}
\makeatother

\newcommand{\scrbar}[1]{\overline{\mathcal{#1}}}
\newcommand{\scrhat}[1]{\widehat{\mathcal{#1}}}
\newcommand{\scrtl}[1]{\widetilde{\mathcal{#1}}}

\ExplSyntaxOn

\cs_new_protected:Nn \bamboo_define:nnnnnN
{
	\cs_new_protected:cpx { #3 #1 #4 } { \exp_not:N #5{#6{#2}} }
}

\int_step_inline:nnn { `A } { `Z }
{
	\bamboo_define:nnnnnN
	{ \char_generate:nn { #1 } { 11 } }
	{ \char_generate:nn { #1 } { 11 } }
	{ bl }
	{ }
	{ \mymathbold }
	\use:n
	
	\bamboo_define:nnnnnN
	{ \char_generate:nn { #1 } { 11 } }
	{ \char_generate:nn { #1 } { 11 } }
	{ scr }
	{ }
	{ \mathcal }
	\use:n
	
	\bamboo_define:nnnnnN
	{ \char_generate:nn { #1 } { 11 } }
	{ \char_generate:nn { #1 } { 11 } }
	{ }
	{ hat }
	{ \widehat }
	\use:n
	
	\bamboo_define:nnnnnN
	{ \char_generate:nn { #1 } { 11 } }
	{ \char_generate:nn { #1 } { 11 } }
	{ scr }
	{ hat }
	{ \scrhat }
	\use:n
	
	\bamboo_define:nnnnnN
	{ \char_generate:nn { #1 } { 11 } }
	{ \char_generate:nn { #1 } { 11 } }
	{ }
	{ br }
	{ \overline }
	\use:n
	
	\bamboo_define:nnnnnN
	{ \char_generate:nn { #1 } { 11 } }
	{ \char_generate:nn { #1 } { 11 } }
	{ }
	{ tl }
	{ \widetilde }
	\use:n
	
	\bamboo_define:nnnnnN
	{ \char_generate:nn { #1 } { 11 } }
	{ \char_generate:nn { #1 } { 11 } }
	{ scr }
	{ br }
	{ \scrbar }
	\use:n
	
	\bamboo_define:nnnnnN
	{ \char_generate:nn { #1 } { 11 } }
	{ \char_generate:nn { #1 } { 11 } }
	{ scr }
	{ tl }
	{ \scrtl }
	\use:n
	
	\bamboo_define:nnnnnN
	{ \char_generate:nn { #1 } { 11 } }
	{ \char_generate:nn { #1 } { 11 } }
	{ }
	{ dt }
	{ \dot}
	\use:n
}
\int_step_inline:nnn { `a } { `z }
{
	\bamboo_define:nnnnnN
	{ \char_generate:nn { #1 } { 11 } }
	{ \char_generate:nn { #1 } { 11 } }
	{ bl }
	{ }
	{ \mymathbold }
	\use:n
	
	\bamboo_define:nnnnnN
	{ \char_generate:nn { #1 } { 11 } }
	{ \char_generate:nn { #1 } { 11 } }
	{ }
	{ hat }
	{ \hat }
	\use:n
	
	\bamboo_define:nnnnnN
	{ \char_generate:nn { #1 } { 11 } }
	{ \char_generate:nn { #1 } { 11 } }
	{ }
	{ br }
	{ \bar }
	\use:n
	
	\bamboo_define:nnnnnN
	{ \char_generate:nn { #1 } { 11 } }
	{ \char_generate:nn { #1 } { 11 } }
	{ }
	{ tl }
	{ \tilde }
	\use:n                            % just the argument
	
	\bamboo_define:nnnnnN
	{ \char_generate:nn { #1 } { 11 } }
	{ \char_generate:nn { #1 } { 11 } }
	{ }
	{ dt }
	{ \dot}
	\use:n
}
\clist_map_inline:nn
{
	Gamma,Delta,Theta,Lambda,Xi,Pi,Sigma,Phi,Psi,Omega
}
{
	\bamboo_define:nnnnnN
	{ #1 }
	{ #1 }
	{ bl }
	{ }
	{ \mymathbold }
	\use:c
	
	\bamboo_define:nnnnnN
	{ #1 }
	{ #1 }
	{ op }
	{ }
	{ \op }
	\use:c
	
	\bamboo_define:nnnnnN
	{ #1 }
	{ #1 }
	{ scr }
	{ }
	{ \mathcal }
	\use:c
	
	\bamboo_define:nnnnnN
	{ #1 }
	{ #1 }
	{ }
	{ hat }
	{ \widehat }
	\use:c
	
	\bamboo_define:nnnnnN
	{ #1 }
	{ #1 }
	{ }
	{ br }
	{ \overline }
	\use:c
	
	\bamboo_define:nnnnnN
	{ #1 }
	{ #1 }
	{ }
	{ tl }
	{ \widetilde }
	\use:c
	
	\bamboo_define:nnnnnN
	{ #1 }
	{ #1 }
	{ }
	{ dt }
	{ \dot}
	\use:c
	
}
\clist_map_inline:nn
{
	alpha,beta,gamma,delta,epsilon,zeta,eta,theta,iota,kappa,
	lambda,mu,nu,xi,pi,rho,sigma,tau,phi,chi,psi,omega
}
{
	\bamboo_define:nnnnnN
	{ #1 }
	{ #1 }
	{ bl }
	{ }
	{ \mymathbold }
	\use:c
	
	\bamboo_define:nnnnnN
	{ #1 }
	{ #1 }
	{ }
	{ hat }
	{ \hat }
	\use:c
	
	\bamboo_define:nnnnnN
	{ #1 }
	{ #1 }
	{ }
	{ br }
	{ \bar }
	\use:c
	
	\bamboo_define:nnnnnN
	{ #1 }
	{ #1 }
	{ }
	{ tl }
	{ \tilde }
	\use:c
	
	\bamboo_define:nnnnnN
	{ #1 }
	{ #1 }
	{ }
	{ dt }
	{ \dot}
	\use:c
}

\ExplSyntaxOff

\DeclareMathOperator{\E}{\mathbf{E}}

\newcommand{\rank}{\operatorname{rank}}

\newcommand{\tr}{\operatorname{tr}}

\newcommand{\diag}{\operatorname{diag}}
\newcommand{\ddiag}{\operatorname{ddiag}}

\DeclarePairedDelimiter{\norm}{\lVert}{\rVert}

\DeclarePairedDelimiter{\abs}{\lvert}{\rvert}

\DeclarePairedDelimiter{\parens}{(}{)}
\DeclarePairedDelimiter{\brackets}{[}{]}
\DeclarePairedDelimiterX{\ip}[2]{\langle}{\rangle}{#1,#2}

\DeclarePairedDelimiterXPP{\normsub}[2]{}{\lVert}{\rVert}{_{#2}}{#1}
\DeclarePairedDelimiterXPP{\ipsub}[3]{}{\langle}{\rangle}{_{#3}}{#1,#2}

\DeclarePairedDelimiterXPP{\ipHS}[2]{}{\langle}{\rangle}{_{\mathrm{HS}}}{#1, #2}
\DeclarePairedDelimiterXPP{\normHS}[1]{}{\lVert}{\rVert}{_{\mathrm{HS}}}{#1}

\DeclarePairedDelimiterXPP{\ipF}[2]{}{\langle}{\rangle}{_{\mathrm{F}}}{#1, #2}
\DeclarePairedDelimiterXPP{\normF}[1]{}{\lVert}{\rVert}{_{\mathrm{F}}}{#1}

\DeclarePairedDelimiterXPP{\dkl}[2]{\operatorname{D_{KL}}}{(}{)}{}{#1 \: \delimsize\Vert \: #2}

\DeclarePairedDelimiterXPP{\restr}[2]{}{{}}{\vert}{_{#2}}{#1}

\newcommand{\ones}{\mathbf{1}}

\newcommand{\R}{\mathbf{R}}
\newcommand{\C}{\mathbf{C}}
\newcommand{\Z}{\mathbf{Z}}
\newcommand{\F}{\mathbf{F}}

\newcommand{\range}{\operatorname{range}}
\newcommand{\spn}{\operatorname{span}}

\newcommand{\poissondist}{\operatorname{Poisson}}

\newcommand{\negqquad}{\mspace{-36mu}}

\theoremstyle{plain}%
\newtheorem{theorem}{Theorem}%  meant for continuous numbers
\newtheorem{lemma}{Lemma}% 
\theoremstyle{definition}%
\newtheorem{assumption}{Assumption}%
\Crefname{assumption}{Assumption}{Assumptions}

\newcommand{\fundingack}{This work was supported in part by the Swiss State Secretariat for Education, Research and Innovation (SERI) under contract MB22.00027 during the author's time at EPFL.}

\begin{document}

\title{Phase retrieval via overparametrized nonconvex optimization:\\nonsmooth amplitude loss landscapes}
%\author{Andrew D.\ McRae\thanks{The author is with CERMICS, CNRS, ENPC, Institut Polytechnique de Paris, Marne-la-Vallée, France. E-mail: \href{mailto:andrew.mcrae@enpc.fr}{andrew.mcrae@enpc.fr}.}}

\author{Andrew D.\ McRae%
	\ifMS \thanks{Manuscript received November 24, 2025; revised April 27, 2026. \fundingack{}}\fi %
	\thanks{The author is with CERMICS, CNRS, ENPC, Institut Polytechnique de Paris, Marne-la-Vallée, France. E-mail: \href{mailto:andrew.mcrae@enpc.fr}{andrew.mcrae@enpc.fr}.
	}%
}

\maketitle

\ifMS
\markboth{Journal of \LaTeX\ Class Files,~Vol.~14, No.~8, August~2015}%
{Shell \MakeLowercase{\textit{et al.}}: Bare Demo of IEEEtran.cls for IEEE Journals}
\fi

\begin{abstract}
	We study nonconvex optimization for phase retrieval and the more general problem of semidefinite low-rank matrix sensing;
	in particular, we focus on the global nonconvex landscape of overparametrized versions of the nonsmooth amplitude least-squares loss as well as a smooth reformulation of this loss based on the PhaseCut approach.
	We first give a general, deterministic result on properties of second-order critical points for a general class of loss functions;
	we then specialize this result to the nonsmooth amplitude loss and, additionally, prove nearly identical results for a smooth reformulation (similar to PhaseCut) as a synchronization problem over spheres.
	Finally, we show the usefulness of these tools by proving high-probability landscape guarantees in two settings:
	(1) phase retrieval with isotropic sub-Gaussian measurements,
	and (2) phase retrieval in a general (possibly infinite-dimensional) Hilbert space with Gaussian measurements.
	In both cases, our results give state-of-the-art and statistically optimal guarantees with only a constant amount of overparametrization (in the well-studied case of isotropic sub-Gaussian measurements, such statistical guarantees had previously required greater degrees of overparametrization/relaxation);
	this demonstrates the potential of overparametrized nonconvex optimization as a principled and scalable algorithmic approach to phase retrieval.
\end{abstract}

\ifMS
\begin{IEEEkeywords}
	Phase retrieval, low-rank matrix sensing, nonconvex optimization landscapes, Burer-Monteiro factorization.
\end{IEEEkeywords}
\fi

\section{Introduction}
\label{sec:intro}
In phase retrieval, we want to estimate a vector $x_*$ from (possibly noisy) measurements of the form $y_i \approx \abs{\ip{f_i}{x_*}}$ for $i = 1, \dots, n$,
where $x_*$ and $f_1, \dots, f_n$ are vectors in $\F^d$;
$\F$ is the field of either real numbers ($\F = \R$) or complex numbers ($\F = \C$),
and $\ip{\cdot}{\cdot}$ denotes the standard Euclidean inner (dot) product on $\F^d$.
This is a critical problem in, for example, computational imaging,
where the linear part $\ip{f_i}{x_*}$ often represents the (complex) amplitude of an electromagnetic wave at a point after some linear diffraction process,
but we can only physically measure the strength (amplitude) of the wave and not its phase.

We can write this formally as the generalized linear model
\begin{equation}
	\label{eq:pr_model_abs}
	\begin{aligned}
		y_i &= \abs{\ip{f_i}{x_*}} + \varepsilon_i, \quad i = 1,\dots, n, \quad \text{or} \\
		y &= \abs{F x_*} + \varepsilon \in \R^n,
	\end{aligned}
\end{equation}
where $\varepsilon_1, \dots, \varepsilon_n \in \R$ represent measurement noise, we denote
\[
	F \coloneqq \begin{bmatrix*} f_1^* \\ \vdots \\ f_n^* \end{bmatrix*} \in \F^{n \times d}, \quad y = \begin{bmatrix*} y_1 \\ \vdots \\ y_n \end{bmatrix*}, \quad \text{and} \quad \varepsilon = \begin{bmatrix*} \varepsilon_1 \\ \vdots \\ \varepsilon_n \end{bmatrix*},
\]
and the absolute value of a vector is understood elementwise.%
\footnote{By the usual convention on the complex inner product, $(F x_*)_i = \ip{x_*}{f_i} = \overline{\ip{f_i}{x_*}}$, but we find the ordering $\ip{x_*}{f_i}$ unnatural. We will endeavor to be precise in the rare cases where this complex conjugation matters.}

An equivalent version of this is
\begin{equation}
	\label{eq:pr_model_quad}
	\begin{aligned}
		y_i^2 &= \abs{\ip{f_i}{x_*}}^2 + \xi_i, \quad i = 1,\dots, n,
		\quad \text{or} \\
		y^2 &= \abs{F X}^2 + \xi,
	\end{aligned}
\end{equation}
where the square of a vector is understood elementwise,
and $\xi$ again represents noise (note that $\xi_i = 2 \abs{\ip{f_i}{x_*}} \varepsilon_i + \varepsilon_i^2$).

%\subsection{Properties}

\subsection{General semidefinite matrix sensing}
\label{sec:intro_ms}
A convenient analysis framework for phase retrieval is low-rank matrix sensing.
This arises from the fact that $\abs{\ip{f_i}{x_*}}^2 = \ip{f_i f_i^*}{x_* x_*^*}$,
where $\ip{\cdot}{\cdot}$ applied to matrices is the elementwise Euclidean (or Frobenius, or trace) inner product; this reformulation is sometimes called the ``lifting'' trick.
Our problem is thus equivalent to that of recovering a rank-1 matrix $Z_* = x_* x_*^*$ from (noisy) linear measurements of the form $\ip{A_i}{Z_*}$,
where $A_i = f_i f_i^*$.

As in the paper \cite{McRae2026a}, of which the present work is a continuation, much of our theoretical analysis of phase retrieval generalizes to the problem of estimating
a rank-$r$ positive semidefinite (PSD) matrix $Z_*$ from (possibly noisy) measurements $\approx \ip{A_i}{Z_*}$ for matrices $A_1, \dots, A_n$ that are also PSD.
We can write $Z_* = X_* X_*^*$ for some matrix $X_* \in \F^{d \times r}$.
This generalization includes the important case of ordinary phase retrieval when the signal $x_*$ is real ($\F = \R$), but the measurement vectors $f_i$ are complex;
we can then take the rank-2 real measurement matrix $A_i = \real(f_i f_i^*)$.

With this setup, following the notation of \cite{Balan2016,Balan2022},
we denote, for arbitrary $r' \geq 1$, the maps $\alpha, \beta \colon \F^{d \times r'} \to \R^n$ defined by
\begin{align*}
	\alpha(X) &\coloneqq \begin{bmatrix*} \ip{A_1}{X X^*}^{1/2} \\ \vdots \\ \ip{A_n}{X X^*}^{1/2} \end{bmatrix*} \quad \text{and} \\
	\beta(X) &\coloneqq \begin{bmatrix*} \ip{A_1}{X X^*} \\ \vdots \\ \ip{A_n}{X X^*} \end{bmatrix*} = \alpha^2(X).
\end{align*}
In the phase retrieval case where $A_i = f_i f_i^*$ for vectors $f_1, \dots f_n \in \F^d$,
we have
\[
	\alpha(X) = \abs{F X}, \quad \text{and} \quad \beta(X) = \abs{F X}^2,
\]
where the absolute value is interpreted as the vector of row norms.

The model \eqref{eq:pr_model_abs} thus generalizes to
\begin{equation}
	\label{eq:gen_model_sqrt}
	\begin{aligned}
		y_i &= \ip{A_i}{X_* X_*^*}^{1/2} + \varepsilon_i, \quad i = 1,\dots, n, \quad \text{or} \\
		y &= \alpha(X_*) + \varepsilon.
	\end{aligned}
\end{equation}
Equivalently, we can write, similarly to the quadratic model \eqref{eq:pr_model_quad},
\begin{equation}
	\label{eq:gen_model_quad}
	\begin{aligned}
		y_i^2 &= \ip{A_i}{X_* X_*^*} + \xi_i, \quad i = 1,\dots, n,
	\quad \text{or} \\
	y^2 &= \beta(X_*) + \xi.
	\end{aligned}
\end{equation}
We will, throughout this paper, assume that $y_1, \dots, y_n \geq 0$.
This is natural due to the inherent nonnegativity of the models,
and, in practice, negative values due to measurement error can be truncated to zero.
\subsection{Recovery via optimization}
\label{sec:intro_opt}
Given phaseless measurements $y_i \approx \abs{\ip{f_i}{x_*}}$ for $i = 1, \dots, n$,
an important question is how to estimate the vector of interest $x_*$.
There are many ways to do this;
see \Cref{sec:relwork} for some pointers toward the vast literature on this topic.
Many approaches use optimization to fit a variable $x \in \F^d$ to the data.
For example, one simple method is the following least-squares problem based on the quadratic model \eqref{eq:pr_model_quad}:
\begin{equation}
	\label{eq:unreg_pr_quartic}
	\min_{x \in \F^d}~\frac{1}{2n} \sum_{i=1}^n ( \abs{\ip{f_i}{x}}^2 - y_i^2 )^2.
\end{equation}
The generalized version of this, based on \eqref{eq:gen_model_quad}, is
\begin{equation}
	\label{eq:unreg_gen_quartic}
	\min_{X \in \F^{d \times p}}~\frac{1}{2n} \underbrace{\sum_{i=1}^n ( \ip{A_i}{X X^*} - y_i^2 )^2}_{= \norm{\beta(X) - y^2}^2},
\end{equation}
where $p$ represents the (maximum) rank of the matrices over which we are searching (typically, we would set $p = r$ if this is known).
This approach is widely used and studied in large part because it is smooth (the objective is a quartic polynomial)
and because it arises naturally from the low-rank matrix sensing framework described in \Cref{sec:intro_ms}.

Another least-squares approach, which will be the main focus of this paper, is based on the model \eqref{eq:pr_model_abs}:
\begin{equation}
	\label{eq:unreg_pr_sqrt}
	\min_{x \in \F^d}~\frac{1}{n} \sum_{i=1}^n ( \abs{\ip{f_i}{x}} - y_i )^2.
\end{equation}
%There are several practical and principled reasons to prefer this over the quartic problem \eqref{eq:unreg_pr_quartic};
%see the next subsection for discussion.
This is often known as an ``amplitude'' formulation,
as, in imaging problems (as discussed above),
$\abs{\ip{f_i}{x_*}}$ is often the amplitude of an electromagnetic wave at a point;
in contrast, the problem \eqref{eq:unreg_pr_quartic} is an ``intensity'' formulation, as the energy in the wave is proportional to $\abs{\ip{f_i}{x}}^2$.
Generalized to low-rank semidefinite matrix sensing, this becomes
\begin{equation}
	\label{eq:unreg_gen_sqrt}
	\min_{X \in \F^{d \times p}}~\frac{1}{n} \underbrace{\sum_{i=1}^n ( \ip{A_i}{X X^*}^{1/2} - y_i )^2}_{= \norm{\alpha(X) - y}^2}.
\end{equation}
There are several reasons to prefer this loss function over the quartic one of \eqref{eq:unreg_pr_quartic} and \eqref{eq:unreg_gen_quartic}:
\begin{itemize}
	\item In practice, this loss function often works better in terms of reconstruction error, algorithmic convergence rate, and the number of measurements required.
	See, for example, \cite{Yeh2015,Wang2018b,Zhang2016} for empirical evidence and further discussion.
	\item This loss function is a good (local) approximation to the Poisson (negative) log-likelihood function,
	which is a natural choice in many imaging problems; see, for example, \cite{Seifert2023}.
	It furthermore is finite everywhere, unlike the Poisson log-likelihood function.
	\item For obtaining theoretical guarantees for ordinary phase retrieval, the fact that the objective function in \eqref{eq:unreg_pr_quartic} and \eqref{eq:unreg_gen_quartic} is quartic in the measurement vectors $\{f_i\}_i$ makes the landscape analysis quite sensitive to large measurements.
	For example, even with Gaussian measurements (that is, the vectors $f_1, \dots, f_n$ are independent standard Gaussian vectors),
	heavy tails lead to considerable difficulty in the analysis and result in some logarithmic factors in the results we can obtain
	(the paper \cite{Liu2024} suggests this difficulty is inevitable for this loss function, at least without using completely different analysis tools).
	Our statistical guarantees for \eqref{eq:unreg_gen_sqrt} in \Cref{sec:stats} will indeed avoid this issue.
%	\item The nonsmooth problems \eqref{eq:unreg_pr_sqrt} and \eqref{eq:unreg_gen_sqrt} are amenable to reformulation as smooth \emph{synchronization} problems over the unknown measurement phases/directions, as in the PhaseCut approach of \cite{Waldspurger2015} (a special case of this approach also appeared in \cite{Halsted2022}).
%	Indeed, we explore this in \Cref{sec:phasecut} of the present paper.
\end{itemize}
Both formulations \eqref{eq:unreg_gen_quartic} and \eqref{eq:unreg_gen_sqrt} are nonconvex and thus possibly have spurious local optima.
The problem \eqref{eq:unreg_pr_sqrt} has yet another difficulty:
it is \emph{nonsmooth} due to the square root
(in \eqref{eq:unreg_pr_sqrt}, this appears in the unsquared absolute value).
Nonsmoothness makes practical optimization and theoretical analysis more delicate
as important quantities such as the gradient and Hessian may not exist everywhere.

One way around this nonsmoothness is to try to reformulate the problem to be smooth.
For example, in the ordinary phase retrieval problem \eqref{eq:unreg_pr_sqrt} in the complex case $\F = \C$,
we can replace each absolute value of the form $\abs{a}$ by a product $a s$, where $s \in \C$ is an auxiliary unit-modulus complex number
(we describe this in more detail in \Cref{sec:phasecut}).
On the one hand, this gives us a smooth problem as the unit circle in $\C$ is a smooth manifold.
On the other hand, the problem is still nonconvex, and it is further complicated by a manifold constraint.
In the paper \cite{Waldspurger2015} which suggested this formulation (``PhaseCut''),
the problem was relaxed to a semidefinite program, but we want to avoid this as it greatly increases the size and computational burden of the problem.
Thus there remains considerable difficulty in understanding how to solve practically and efficiently the nonconvex and nonsmooth problems \eqref{eq:unreg_pr_sqrt} and \eqref{eq:unreg_gen_sqrt}.

\subsection{Summary of main contributions}
In this paper, we are interested in theoretical guarantees for practical algorithms;
can we prove that, under reasonable statistical assumptions on the measurements, there is a computationally practical algorithm that will (e.g., with high probability) return a good estimate of the ground truth?
Most existing theoretical results fall into one of the following categories (see \Cref{sec:relwork} for more details):
\begin{itemize}
	\item Many study direct optimization approaches to nonconvex problems similar to those above.
	The results that give statistically optimal results in general only apply to the ordinary phase retrieval problem and make quite idealistic assumptions like Gaussian measurements.
	The results applying to more general problems and measurements are typically statistically suboptimal in terms of how large they require $n$, the number of measurements, to be.
	\item Most state-of-the-art statistical guarantees, outside of ordinary phase retrieval and Gaussian measurements,
	are for semidefinite relaxations (replacing the matrix $X X^*$ by an arbitrary PSD matrix in \eqref{eq:unreg_gen_quartic} or a similar problem, thus making the problem convex).
	However, this greatly increases the computational complexity compared to direct nonconvex approaches,
	as this requires optimization over $\approx d^2$ variables.
\end{itemize}
The recent paper \cite{McRae2026a}, of which the present work is a continuation, attempted to address these issues via nonconvex partial relaxations of the quartic problem \eqref{eq:unreg_gen_quartic}:
we relax (or overparametrize) the problem by setting the optimization rank $p$ to be larger than the ground truth rank $r$ (in ordinary phase retrieval, $r = 1$).
That paper showed that, in a variety of settings, this rank overparametrization allows the nonconvex problem to have a benign landscape,
that is, every local optimum is global (or at least is a good estimator in a statistical sense).
Since many local optimization algorithms can find such local optima (or, more technically, second-order critical points; see \Cref{sec:determ_direct} for further discussion), this gives theoretical guarantees for computationally practical algorithms.
When $p$ is a logarithmic multiple of the true rank $r$,
we can obtain statistically optimal results that had previously only been shown for semidefinite programming approaches;
this can greatly decrease the computational burden for theoretically principled algorithms when the dimension $d$ is large.

This present work seeks to expand on \cite{McRae2026a} by considering alternative loss functions to \eqref{eq:unreg_gen_quartic},
primarily the nonsmooth loss in \eqref{eq:unreg_gen_sqrt}.
As discussed above in \Cref{sec:intro_opt},
there are several practical and principled reasons to prefer this formulation.
However, the nonsmoothness greatly complicates the analysis.
In \Cref{sec:determ_direct}, we develop deterministic landscape analysis tools for a general class of possibly nonsmooth loss functions.
This class includes the (smooth) quartic loss of \eqref{eq:unreg_gen_quartic} (we obtain \Cref{thm:landscape_quad}, which is a refinement of \cite[Lem.~1]{McRae2026a}),
the nonsmooth loss of \eqref{eq:unreg_gen_sqrt} (\Cref{thm:landscape_sqrt}),
and the nonsmooth (and possibly infinite-valued) Poisson log-likelihood loss (\Cref{thm:landscape_poiss}).
These results are deterministic in the problem data $(A_i, y_i)_i$ and are stepping stones to further statistical guarantees.

In parallel, in \Cref{sec:phasecut},
we study an alternative smooth formulation of \eqref{eq:unreg_gen_sqrt} as a optimization problem over spheres based on the PhaseCut approach of \cite{Waldspurger2015}.
This is interesting practically because it avoids all issues of nonsmoothness;
even if we know a nonsmooth nonconvex landscape is benign in some sense,
algorithm design is more complicated, and there are few algorithmic convergence guarantees in the literature.
Furthermore, this smooth formulation connects the semidefinite matrix sensing problem to a class of \emph{synchronization} problems for which there has recently been much interest and many nonconvex landscape results in the optimization, control theory, and dynamical systems communities.
We give (\Cref{thm:landscape_phasecut}) deterministic landscape guarantees nearly identical to those (\Cref{thm:landscape_sqrt}) for the original nonsmooth problem \eqref{eq:unreg_gen_sqrt}.

We then demonstrate, in \Cref{sec:stats}, the power of our deterministic landscape tools for obtaining statistical guarantees.
For clarity and brevity we focus on the rank-$1$ ordinary phase retrieval problem,
although our deterministic tools apply to general semidefinite matrix sensing problems.
In \Cref{sec:subG}, we give a result (\Cref{thm:finite_subG}, parallel to \cite[Thm.~3]{McRae2026a})
for phase retrieval with general sub-Gaussian measurements;
only a constant (rather than logarithmic as in \cite{McRae2026a}) level of overparametrization is necessary to ensure statistically optimal recovery (in terms of error and sample complexity) with high probability.
This gives theoretical justification to using the nonsmooth loss function in \eqref{eq:unreg_gen_sqrt} as well as its smooth reformulation over spheres.

Finally, we show that, with some additional regularization,
we can obtain theoretical guarantees for infinite-dimensional phase retrieval:
our \Cref{thm:infdim} in \Cref{sec:infdim} gives optimal recovery results for phase retrieval with Gaussian measurements in a general Hilbert space.
Although our assumptions are idealized,
this result is a potential step toward principled and practical algorithms for phase retrieval in infinite dimensions (particularly in reproducing kernel Hilbert spaces),
which is a topic of significant recent interest (see \Cref{sec:relwork}).

\subsection{Notation}
For reference, we compile here some notation that will be used throughout the paper.
Certain other more specialized notation will be defined later as needed.

For an integer $m \geq 1$, we denote by $I_m$ the $m \times m$ identity matrix.
For matrices $A, B$ of equal size (including ordinary vectors), we denote by $\ip{A}{B} = \tr(B^* A)$
their elementwise Euclidean (Frobenius) inner product.
For a matrix $A$, we denote its Frobenius (elementwise $\ell_2$), operator, and nuclear norms respectively by $\norm{A}$, $\opnorm{A}$, and $\nucnorm{A}$.

For a vector $v \in \F^m$ (with $\F = \R$ or $\F = \C$),
we denote its Euclidean ($\ell_2$), $\ell_\infty$, and $\ell_1$ norms respectively by $\norm{v}$, $\norm{v}_\infty$, and $\norm{v}_1$.
We will also write $v^2 \in \F^m$ to mean the elementwise square of $v$.
Its elementwise absolute value is denoted by $\abs{v} \in \R^m$.
We will occasionally overload this notation to apply to a matrix $V \in \F^{m \times p}$:
$\abs{V} \in \R^m$ is the vector whose elements are the ($\ell_2$) norms of the rows of $V$ (which are vectors in $\F^p$).

We overload the notation $\diag$ as follows: for a vector $v \in \F^m$, $\diag(v) \in \F^{m \times m}$ is the diagonal matrix whose diagonal elements are the elements of $v$.
For $A \in \F^{m \times m}$, $\diag(A) \in \F^m$ is the vector of the diagonal elements of $A$.

We write the real and imaginary parts of $a \in \C$ as $\real(a)$ and $\imag(a)$ respectively,
and we write its complex conjugate as $\abr$.
We also overload this notation to apply elementwise if $a$ is a vector or matrix.

We will often informally write $a \lesssim b$ to mean that $a \leq C b$ for some unspecified constant $C > 0$ that does not depend on problem parameters such as the dimension or number of measurements.
$a \gtrsim b$ means $b \lesssim a$.
$a \approx b$ means $a \lesssim b$ and $a \gtrsim b$ simultaneously.

\section{Related work}
\label{sec:relwork}
In this section, we review the most relevant literature on phase retrieval and the various optimization approaches to these problems.
For further reading on the importance (particularly in imaging) of the phase retrieval problem and an overview of the many algorithms and theoretical analyses developed for it,
see the surveys \cite{Shechtman2015,Fannjiang2020,Dong2023}.

There are relatively few results on global nonconvex optimization landscapes for phase retrieval.
For the quartic problem \eqref{eq:unreg_pr_quartic},
the first global landscape guarantee was given in \cite{Sun2018},
which was then improved in \cite{Cai2023}; this result (also recovered in \cite{McRae2026a}) ensures a globally benign landscape with $n \gtrsim d \log d$ Gaussian measurements.
It is unlikely that we can do any better for this quartic loss function without fundamentally different arguments;
the paper \cite{Liu2024} shows that, in the regime $n \lesssim d \log d$, the landscape becomes more delicate,
so standard analysis methods fail (nevertheless, the paper \cite{SaraoMannelli2020} provides numerical evidence and heuristic arguments that, indeed, the landscape is benign whenever $n \gtrsim d$).
To avoid this logarithmic factor (which is statistically suboptimal),
a different approach is needed.
In \cite{McRae2026a}, we showed that overparametrization (by at most a logarithmic factor, i.e., $p \approx \log d$) of the quartic formulation gets around this statistical suboptimality.
Another option is to consider a different loss function,
which we do in this paper.

To our knowledge, there are no existing works which study the global landscape of nonsmooth problems of the form \eqref{eq:unreg_pr_sqrt} or \eqref{eq:unreg_gen_sqrt}.
The paper \cite{Davis2020} studies a different nonsmooth formulation,
but its results only characterize (first-order) critical points rather than local optima or second-order critical points.
To get around the issue of nonsmoothness,
a recent series of papers \cite{Li2020a,Cai2021,Cai2022,Cai2022a} has studied the landscape of \emph{smoothed} versions of \eqref{eq:unreg_pr_sqrt}.
These works indeed show that,
with these alternative loss functions,
one obtains a benign landscape with $n \gtrsim d$ Gaussian measurements.
Interestingly, this applies to the non-overparametrized case,
whereas our positive landscape results for the non-smoothed version do require some overparametrization ($p > 1$).
It is unclear if this is a fundamental drawback of the unsmoothed loss function or merely of our analysis.

There has been far more work on understanding how well we can solve problems like \eqref{eq:unreg_pr_sqrt} by a two-stage algorithm consisting of an initialization (e.g., by spectral methods) and then refinement by local optimization (e.g., by some form of gradient descent)
\cite{Xiao2021,Li2025,Wang2018b,Wang2018a,Zhang2017a,Luo2020}.
Some of these works (\cite{Xiao2021,Luo2020}) consider smoothed versions of \eqref{eq:unreg_pr_sqrt} similar to those mentioned in the previous paragraph.
In the case of Gaussian measurements, these papers show that $n \gtrsim d$ measurements suffice.
Furthermore, the recent paper \cite{Li2025} provides guarantees for a physically-inspired coded diffraction pattern measurement model with optimal sample complexity.

Two-stage algorithms and analyses for the quartic formulation in \eqref{eq:unreg_pr_quartic} have also been well-studied;
see, for example, \cite{Ma2020,Peng2024,Gao2021a,Chen2019b} for state-of-the-art results and further reading.
The Poisson maximum likelihood loss function which we briefly consider in \Cref{sec:determ_poisson} is also a popular practical choice (again, see, for example, \cite{Yeh2015,Seifert2023}),
but the loss function is difficult to analyze theoretically,
so guarantees are few; a recent work in this direction is \cite{Gao2025}.

The sub-Gaussian measurement model described in \Cref{sec:subG} has seen much recent work,
as it is one of the most general models for which statistically (near-)optimal recovery guarantees have been proved.
Several algorithms and theoretical analyses under variants of the assumptions of our \Cref{thm:finite_subG} have been developed in the papers \cite{Eldar2014,Krahmer2018,Krahmer2020,Gao2021a,Peng2024}.
In particular, the paper \cite{Krahmer2020} shows, under conditions nearly identical to our \Cref{thm:finite_subG},
that a semidefinite programming approach can recover the ground truth with optimal sample complexity (i.e., requiring only $n \gtrsim d$).
The best results we are aware of for nonconvex approaches (prior to \cite{McRae2026a}) are those of \cite{Gao2021a,Peng2024}, which show that, under similar sub-Gaussian measurement assumptions,
a two-stage initialization + gradient descent algorithm
can recover the ground truth if $n \gtrsim d \log^2 d$.

Phase retrieval in infinite dimensions (and, particularly, in function spaces) has seen much recent interest.
However, the results we are aware of primarily consider questions of identifiability and stability rather than specific reconstruction algorithms,
so they are not directly comparable to our results.
See, for example, the survey \cite{Grohs2020} and the more recent works \cite{Alharbi2024,Freeman2024} for further background and references.
One work that does give guarantees for an algorithm is \cite{Chen2020a};
another more recent work that specifically considers reproducing kernel Hilbert spaces is \cite{Fuehr2025}.
However, again, their results are not readily comparable to ours.

Overparametrization and benign landscapes have also been extensively studied in synchronization problems,
of which our PhaseCut-based formulation in \Cref{sec:phasecut} is an example.
These problems have much broader application in areas such as robotics, dynamical systems, signal processing, and graph clustering.
For an introduction and further references, see, for example, \cite{Ling2023a,Ling2025,Endor2026,McRae2025preprinta}.
The basic geometric tools we use in \Cref{sec:phasecut} are taken directly from this literature.
However, our problem instances have considerably different properties than those studied previously,
so new methods are also needed.

These themes of benign landscapes and overparametrization have also been important in low-rank matrix sensing, of which the semidefinite matrix sensing problem described in \Cref{sec:intro_ms} is an instance.
This literature primarily considers quartic formulations like \eqref{eq:unreg_gen_quartic}.
However, most existing results assume a type of \emph{restricted isometry property}
which does not, in general, hold for the measurements that we consider in this paper (getting around this difficulty was a major part of the contribution of \cite{McRae2026a}).
See, for example, \cite{Bi2022,Zhang2025a}, for state-of-the-art results and further references.

\section{Direct optimization: deterministic results}
\label{sec:determ_direct}
In this section, we develop deterministic nonconvex landscape analysis tools for a general class of possibly nonsmooth loss functions.
The nonsmooth problem \eqref{eq:unreg_gen_sqrt} from \Cref{sec:intro} is included,
but our results are more general and thus may be of independent interest.

We consider regularized nonconvex optimization problems of the form
\begin{equation}
	\label{eq:ncvx_gen}
	\begin{gathered}
		\min_{X \in \F^{d \times p}}~L_\lambda(X X^*), \quad \text{where} \\
		L_\lambda(Z) \coloneqq \frac{1}{n} \sum_{i=1}^n \ell(\ip{A_i}{Z}, y_i) + \lambda \tr Z,
	\end{gathered}
\end{equation}
$\ell(b, \upsilon)$ is a loss function\footnote{We write the first argument as $\ip{A_i}{Z}$ rather than $\ip{A_i}{Z}^{1/2}$ to make derivatives cleaner.} that we will specify,
$p$ is a search rank parameter,
and $\lambda \geq 0$ is a regularization parameter.

A semidefinite program (SDP) version of \eqref{eq:ncvx_gen} is
\begin{equation}
	\label{eq:sdp_gen}
	\min_{Z \succeq 0}~L_\lambda(Z).
\end{equation}
If $\ell(b, \upsilon)$ is convex in its first argument, this is a convex SDP.
In this paper, we will consider general $\ell$ with the following properties:
\begin{assumption}
	\label{assump:lossfunc}
	For all $\upsilon \geq 0$, the loss function $\ell(b, \upsilon$) is convex and lower semicontinuous in $b$ on $[0, \infty)$.
	Furthermore, it is finite and twice differentiable in $b$ on $(0, \infty)$.
	We denote the first and second derivatives in $b$ by $\ell'(b, \upsilon)$ and $\ell''(b, \upsilon)$ respectively.
	We furthermore denote
	\[
		\ell'(0, \upsilon) \coloneqq \lim_{b \to 0^+} \ell'(b, \upsilon)
	\]
	which could be $- \infty$.
\end{assumption}

The gradient\footnote{In the complex case, we can compactly define the gradient of a real-valued function $f \colon \C^D \to \R$ at $x \in \C^D$ as the vector $v \in \C^D$ such that $x' \mapsto f(x) + \real( \ip{v}{x' - x} )$ is the best local linear approximation of $f$ around $x$.} of $L_\lambda(Z)$ with respect to the $d \times d$ matrix variable $Z$ is (when it is well-defined)
\begin{equation}
	\label{eq:gradL}
	\nabla L_\lambda(Z)= \frac{1}{n} \sum_{i=1}^n \ell'(\ip{A_i}{Z}, y_i) A_i + \lambda I_d.
\end{equation}

In this paper, we are interested in understanding \emph{second-order critical points} of nonconvex problems like \eqref{eq:ncvx_gen}.
For a twice-differentiable objective, these are points where the gradient with respect to $X$ is zero and the Hessian is positive semidefinite.
However, as our objective is potentially nonsmooth,
we need a generalized notion of second-order criticality.
We say that for a function $f \colon \F^D \to \R$,
$x \in \F^D$ is a second-order critical point if
\[
	\liminf_{x' \to x} \frac{f(x') - f(x)}{\norm{x' - x}^2} \geq 0.
\]
In other words, there is no direction starting from $x$ along which $f$ decreases at least quadratically.
This condition is comparable to standard notions of second-order local optimality in the literature; for example, in our case, this is precisely the condition $d^2(x \mid 0)(w) \geq 0$ of \cite[Thm.~13.24]{Rockafellar1998}.
There is potential ambiguity in the case $f(x) = + \infty$, but, as we will see below (in \Cref{lem:socp_char} and its proof), this is not an issue for our problem.

In the smooth case, there are theoretical results guaranteeing that local search methods such as gradient descent or trust-region algorithms will converge to a second-order critical point;
see, for example, \cite{Curtis2018,Lee2019b}.
Although finding local minima of nonsmooth problems is NP-hard in general \cite{Kornowski2024},
the specific problems described in this section have some nice properties (again, see \Cref{lem:socp_char} below),
so we think it likely that algorithmic guarantees could be obtained.
However, this is beyond the scope of the present paper.
In practice, to avoid questions of nonsmoothness and to ensure numerical stability of gradient-based algorithms, we can replace $\ell(b, \upsilon)$ by
\[
	\ell_\delta(b, \upsilon) \coloneqq \ell(b + \delta, \sqrt{\upsilon^2 + \delta})
\]
for some smoothing parameter $\delta \geq 0$. If $\delta > 0$, our assumptions on $\ell$ ensure that $\ell_\delta(b, \upsilon)$ is finite and twice differentiable in $b$ on $[0, \infty)$ for all $\upsilon \geq 0$.
We could explicitly adapt our theoretical results to account for general $\delta \geq 0$,
but for brevity and simplicity we omit this;
we would typically choose a very small $\delta > 0$, which will have little effect on the results.

The main result of this section is the following lemma:
\begin{lemma}
	\label{lem:landscape_gen}
	Suppose $\ell$ satisfies \Cref{assump:lossfunc}.
	Let $A_1, \dots, A_n \succeq 0$ and $y_1, \dots, y_n \geq 0$,
	and furthermore assume that $A_i \neq 0$ for all $i = 1, \dots, n$.
	Then every second-order critical point $X$ of \eqref{eq:ncvx_gen}
	satisfies the following:
	\begin{itemize}
		\item $L_\lambda(X X^*)$ is finite, and, for all $i$, $\ell'(\ip{A_i}{X X^*}, y_i)$ is finite, so $\nabla L_\lambda(X X^*)$ as given by \eqref{eq:gradL} is well-defined.
		\item $\nabla L_\lambda(X X^*) X = 0$.
		\item For all $Z' \succeq 0$,
		\begin{align*}
			0 &\leq \ip{\nabla L_\lambda(X X^*)}{Z'} \\
			&\quad + \frac{2}{\constF p n} \cdot \sum_{\substack{i=1\\ \mathclap{\ip{A_i}{X X^*} > 0}}}^n \ell''(\ip{A_i}{X X^*}, y_i) \ip{A_i}{X X^*} \ip{A_i}{Z'},
		\end{align*}
		where $\constR = 1$, and $\constC = 2$.
	\end{itemize}
\end{lemma}
We defer the proof to \Cref{sec:determ_proof}.
The assumption that each $y_i \geq 0$ was introduced and discussed in \Cref{sec:intro_ms}.
The assumption that each $A_i$ is nonzero avoids the potential issue of $L_\lambda$ being infinite everywhere.
In practice, we can discard any zero measurement matrices because the corresponding measurements give no useful information.

We can compare this result to what we would expect for the convex problem \eqref{eq:sdp_gen}.
Optimality of $Z \succeq 0$ for \eqref{eq:sdp_gen} is equivalent to $\nabla L_\lambda(Z) Z = 0$ and $\nabla L_\lambda(Z) \succeq 0$.
The criticality property $\nabla L_\lambda(X X^*) X = 0$ clearly implies $\nabla L_\lambda(Z) Z = 0$ for $Z = X X^*$.
However, we do not quite obtain the second condition $\nabla L_\lambda(Z) \succeq 0$;
note that this would be equivalent to $\ip{ \nabla L_\lambda(X X^*) }{Z'} \geq 0$ for all $Z' \succeq 0$.
The additional nonnegative terms involving $\ell''$ in the inequality of \Cref{lem:landscape_gen} are thus the price we pay for the nonconvexity of \eqref{eq:ncvx_gen}.
Increasing the optimization rank $p$ reduces this penalty
(although, as discussed in \cite{McRae2026a}, this does not immediately imply that this is a good thing to do).

For this result to be useful, we must consider specific examples of loss functions $\ell$,
which we do in the following subsections.

\subsection{Example: quartic loss}
We first consider the simple example of the smooth least-squares loss
\[
	\ell(b, \upsilon) = \frac{1}{2}(b - \upsilon^2)^2.
\]
The problem \eqref{eq:ncvx_gen} then becomes (noting that $\tr(X X^*) = \norm{X}^2$)
\begin{equation}
	\label{eq:opt_ncvx_quad}
	\min_{X \in \F^{d \times p}}~\frac{1}{2n} \norm{\beta(X) - y^2}^2 + \lambda \norm{X}^2.
\end{equation}
Note that this is simply a regularized version of \eqref{eq:unreg_gen_quartic};
it is quartic because $\beta(X)$ is quadratic in $X$.
In this case, our \Cref{lem:landscape_gen} recovers the recent result of \cite{McRae2026a}:
\begin{theorem}
	\label{thm:landscape_quad}
	Under the model \eqref{eq:gen_model_quad} for $X_* \in \F^{d \times r}$, 
	every second-order critical point $X$ of \eqref{eq:opt_ncvx_quad} satisfies, for any matrix $R \in \F^{p \times r}$,
	\begin{align*}
		&\frac{1}{n} \norm{\beta(X) - \beta(X_*)}^2 \\
		&\quad\leq \frac{1}{n} \ip{\xi}{\beta(X) - \beta(X_*)} + \lambda(\norm{X_*}^2 - \norm{X}^2) \\
		&\qquad + \frac{2}{\constF p + 2}\parens*{ \frac{1}{n} \ip{y^2}{\beta(X_* - X R)} - \lambda \norm{X_* - X R}^2 },
	\end{align*}
	where $\constF$ is defined in \Cref{lem:landscape_gen}.
\end{theorem}
This is a slight generalization of \cite[Lem.~1]{McRae2026a}, which only considered the unregularized case $\lambda = 0$.
In addition, in the complex case where $\constC = 2$, we have improved the constant.
As the properties of this loss function are thoroughly studied in that paper,
we do not develop them further in the present paper.

\begin{proof}[Proof of \Cref{thm:landscape_quad}]
The derivatives of $\ell$ with respect to the first argument are
\begin{align*}
	\ell'(b, \upsilon) &= b - \upsilon^2, \\
	\ell''(b, \upsilon) &= 1.
\end{align*}
The identity $b \ell''(b, \upsilon) = \ell'(b, \upsilon) + \upsilon^2$ implies, for any $Z' \succeq 0$ and for each $i$,
\begin{align*}
	&\ell''(\ip{A_i}{X X^*}, y_i) \ip{A_i}{X X^*} \ip{A_i}{Z'} \\
	&\quad= \ell'(\ip{A_i}{X X^*}, y_i) \ip{A_i}{Z'} + y_i^2 \ip{A_i}{Z'}.
\end{align*}
As $\ell''(b, \upsilon)$ exists for all $b$ and $\upsilon$, there is no need to separate into cases based on whether $\ip{A_i}{X X^*} = 0$.
We then obtain
\begin{align*}
	&\negqquad \frac{1}{n} \sum_{i=1}^n \ell''(\ip{A_i}{X X^*}, y_i) \ip{A_i}{X X^*} \ip{A_i}{Z'} \\
	&= \frac{1}{n} \sum_{i=1}^n \parens*{ \ell'(\ip{A_i}{X X^*}, y_i) \ip{A_i}{Z'} + y_i^2 \ip{A_i}{Z'} } \\
	&= \ip{\nabla L_\lambda(X X^*)}{Z'} - \lambda \tr Z' + \frac{1}{n} \ip*{ \sum_{i=1}^n y_i^2 A_i }{Z'}.
\end{align*}
From \Cref{lem:landscape_gen} and some algebra, we then obtain
\[
	0 \leq \ip{\nabla L_\lambda(X X^*)}{Z'} + \frac{2}{\constF p + 2}\parens*{ \frac{1}{n} \ip*{ \sum_{i=1}^n y_i^2 A_i }{Z'} - \lambda \tr Z' }.
\]
Finally, we take $Z' = (X_* - X R)(X_* - X R)^*$,
in which case
\[
	\ip*{\sum_{i=1}^n y_i^2 A_i }{Z'} = \ip{y^2}{\beta(X_* - X R)}.
\]
By the fact (also from \Cref{lem:landscape_gen}) that $\nabla L_\lambda(X X^*) X = 0$, we have
\begin{align*}
	&\ip{\nabla L_\lambda(X X^*)}{Z'} \\
	&\quad= \ip{\nabla L_\lambda(X X^*)}{X_* X_*^* - X X^*} \\
	&\quad= \frac{1}{n} \sum_{i=1}^n (\ip{A_i}{X X^*} - y_i^2)\ip{A_i}{X_* X_*^* - X X^*} \\
	&\qquad + \lambda \tr(X_* X_*^* - X X^*) \\
	&\quad= - \frac{1}{n} \sum_{i=1}^n \ip{A_i}{X X^* - X_* X_*^*}^2 + \frac{1}{n} \sum_{i=1}^n \xi_i \ip{A_i}{X X^* - X_* X_*^*} \\
	&\qquad + \lambda (\norm{X_*}^2 - \norm{X}^2) \\
	&\quad= - \frac{1}{n} \norm{\beta(X) - \beta(X_*)}^2 + \frac{1}{n} \ip{\xi}{\beta(X) - \beta(X_*)} \\
	&\qquad + \lambda (\norm{X_*}^2 - \norm{X}^2).
\end{align*}
The result easily follows.
\end{proof}

\subsection{Example: nonsmooth ``amplitude'' loss}
The main choice of loss function we study in this paper is
\[
	\ell(b, \upsilon) \coloneqq (\sqrt{b} - \upsilon)^2,
\]
which is nonsmooth at $b = 0$ whenever $\upsilon > 0$.
The problem \eqref{eq:ncvx_gen} becomes
\begin{equation}
	\label{eq:opt_ncvx_sqrt}
	\min_{X \in \F^{d \times p}}~\frac{1}{n} \norm{\alpha(X) - y}^2 + \lambda \norm{X}^2,
\end{equation}
which is a regularized version of \eqref{eq:unreg_gen_sqrt}.
In this case, \Cref{lem:landscape_gen} implies the following:
\begin{theorem}
	\label{thm:landscape_sqrt}
	Let the search rank $p$ satisfy
	\begin{itemize}
		\item $p \geq 2$ in the real case $\F = \R$, or
		\item $p \geq 1$ in the complex case $\F = \C$.
	\end{itemize}
	Under the model \eqref{eq:gen_model_sqrt} for $X_* \in \F^{d \times r}$, every second-order critical point $X$ of \eqref{eq:opt_ncvx_sqrt} satisfies, for any $R \in \F^{p \times r}$,
	\begin{align*}
		&\frac{1}{n} \norm{\alpha(X) - \alpha(X_*)}^2 \\
		&\quad \leq \frac{2}{n} \ip{\varepsilon}{\alpha(X) - \alpha(X_*)} + \lambda (\norm{X_*}^2 - \norm{X}^2) \\
		&\qquad + \frac{1}{\constF p - 1} \parens*{\frac{1}{n} \norm{\alpha(X_* - XR)}^2 + \lambda \norm{X_* - X R}^2},
	\end{align*}
	where $\constF$ is defined in \Cref{lem:landscape_gen}.
\end{theorem}
We further develop the consequences of this result in \Cref{sec:stats}.
If we can find a global optimum (for example, by setting $p = n$ and using convex semidefinite programming), that optimum will satisfy
\[
	\frac{1}{n} \norm{\alpha(X) - \alpha(X_*)}^2 \leq \frac{2}{n} \ip{\varepsilon}{\alpha(X) - \alpha(X_*)} + \lambda (\norm{X_*}^2 - \norm{X}^2).
\]
The additional term in the bound of \Cref{thm:landscape_sqrt} is, similarly to the general result \Cref{lem:landscape_gen}, the price we pay for only finding a second-order critical point.

\begin{proof}[Proof of \Cref{thm:landscape_sqrt}]
	We have, for $\ell(b, \upsilon) = (\sqrt{b} - \upsilon)^2$,
	\begin{align*}
		\ell'(b, \upsilon) &= 1 - \frac{\upsilon}{\sqrt{b}}, \quad \text{and} \\
		\ell''(b, \upsilon) &= \frac{\upsilon}{2 b^{3/2}} %= \frac{1 - \ell'(b, y)}{2 b}
	\end{align*}
	for $b > 0$.
	The identity $2 b \ell''(b, \upsilon) = 1 - \ell'(b, \upsilon)$
	implies, for any $Z' \succeq 0$ and for each $i$ such that $\ip{A_i}{X X^*} > 0$,
	\begin{align*}
		&\ell''(\ip{A_i}{X X^*}, y_i) \ip{A_i}{X X^*} \ip{A_i}{Z'}
		\\
		&\quad= \frac{1}{2} [ 1 - \ell'(\ip{A_i}{X X^*}, y_i)] \ip{A_i}{Z'}.
	\end{align*}
	\Cref{lem:landscape_gen} ensures that, for all $i$, $\ell'(\ip{A_i}{X X^*}, y_i)$ is finite.
	Furthermore, note that $1 - \ell'(\ip{A_i}{X X^*}, y_i) \geq 0$.
	We then have
	\begin{align*}
		&\frac{1}{n} \cdot \sum_{\substack{i=1\\ \mathclap{\ip{A_i}{X X^*} > 0}}}^n \ell''(\ip{A_i}{X X^*}, y_i) \ip{A_i}{X X^*} \ip{A_i}{Z'} \\
		&\quad\leq \frac{1}{2} \parens*{ \frac{1}{n} \ip*{ \sum_{i=1}^n A_i }{Z'} - \frac{1}{n} \sum_{i=1}^n \ell'(\ip{A_i}{X X^*}, y_i) \ip{A_i}{Z'} } \\
		&\quad= \frac{1}{2} \parens*{ \frac{1}{n} \ip*{ \sum_{i=1}^n A_i }{Z'} - \ip{\nabla L_\lambda(X X^*) }{Z'} + \lambda \tr Z' }.
	\end{align*}
	The inequality of \Cref{lem:landscape_gen} (multiplying by $\constF p$) then implies
	\begin{equation}
		\label{eq:sqrt_lemgen}
		\begin{aligned}
			0 &\leq (\constF p - 1) \ip{\nabla L_{\lambda}(X X^*)}{Z'} \\
			&\quad + \ip*{ \sum_{i=1}^n A_i }{Z'} + \lambda \tr Z'.
		\end{aligned}
	\end{equation}
	Finally, similarly to the proof of \Cref{thm:landscape_quad}, we take $Z' = (X_* - XR)(X_* - XR)^*$, in which case
	\begin{equation}
		\label{eq:sqrt_alpha}
		\ip*{ \sum_{i=1}^n A_i }{Z'} = \norm{\alpha(X_* - XR)}^2,
	\end{equation}
	and we use $\nabla L_{\lambda}(X X^*) X = 0$ to obtain
	\begin{align*}
		&\ip{\nabla L_{\lambda}(X X^*)}{Z'} \\
		&\quad= \ip{\nabla L_{\lambda}(X X^*)}{X_* X_*^* - X X^*} \\
		&\quad= \frac{1}{n} \sum_{i=1}^n \ell'(\ip{A_i}{X X^*}, y_i) \ip{A_i}{X_* X_*^* - X X^*} \\
		&\qquad+ \lambda (\norm{X_*}^2 - \norm{X}^2).
	\end{align*}
	For $i$ such that $\ip{A_i}{X X^*} > 0$, we have, as $y_i \geq 0$,
	\begin{equation}
		\label{eq:sqrt_ineq}
		\begin{aligned}
			&\ell'(\ip{A_i}{X X^*}, y_i) \ip{A_i}{X_* X_*^* - X X^*} \\
			&\quad= \parens*{ 1 - \frac{y_i}{\ip{A_i}{X X^*}^{1/2} }}\ip{A_i}{X_* X_*^* - X X^*} \\
			&\quad\leq \parens*{ 1 - \frac{2 y_i}{\ip{A_i}{X X^*}^{1/2} + \ip{A_i}{X_* X_*^*}^{1/2} }} \\
			&\qquad \times \ip{A_i}{X_* X_*^* - X X^*} \\
			&\quad= (\ip{A_i}{X X^*}^{1/2} - \ip{A_i}{X_* X_*^*}^{1/2} - 2 \varepsilon_i) \\
			&\qquad \times (\ip{A_i}{X_* X_*^*}^{1/2} - \ip{A_i}{X X^*}^{1/2}) \\
			&\quad= -(\ip{A_i}{X X^*}^{1/2} - \ip{A_i}{X_* X_*^*}^{1/2})^2 \\
			&\qquad + 2\varepsilon_i (\ip{A_i}{X X^*}^{1/2} - \ip{A_i}{X_* X_*^*}^{1/2}).
		\end{aligned}
	\end{equation}
	For $i$ such that $\ip{A_i}{X X^*} = 0$, we have (see \Cref{assump:lossfunc})
	\begin{align*}
		\ell'(\ip{A_i}{X X^*} , y_i)
		&= \lim_{b \to 0^+} \ell'(b, y_i) \\
		&= \lim_{b \to 0^+} 1 - \frac{y_i}{\sqrt{b}}.
	\end{align*}
	As this is finite by \Cref{lem:landscape_gen}, we must have $y_i = 0$,
	in which case $\ell'(\ip{A_i}{X X^*} , y_i) = 1$.
	Then, noting that in this case $\varepsilon_i = - \ip{A_i}{X_* X_*^*}^{1/2}$, the first and last expressions in \eqref{eq:sqrt_ineq} are both equal to $\ip{A_i}{X_* X_*^*}$.
	Thus \eqref{eq:sqrt_ineq} holds for all $i$.
	We then obtain
	\begin{align*}
		&\ip{\nabla L_{\lambda}(X X^*)}{Z'} \\
		&\quad\leq - \frac{1}{n} \norm{\alpha(X) - \alpha(X_*)}^2 + \frac{2}{n} \ip{\varepsilon}{\alpha(X) - \alpha(X_*)} \\
		&\qquad + \lambda (\norm{X_*}^2 - \norm{X}^2).
	\end{align*}
	Combining this with \eqref{eq:sqrt_lemgen} and \eqref{eq:sqrt_alpha} and dividing by $\constF p - 1$ completes the proof.
\end{proof}

\subsection{Example: Poisson loss}
\label{sec:determ_poisson}
In this section, we consider the loss function
\[
	\ell(b, \upsilon) \coloneqq b - \upsilon^2 \log b,
\]
with which the general nonconvex problem \eqref{eq:ncvx_gen} becomes
\begin{equation}
	\label{eq:opt_ncvx_poiss}
	\min_{X \in \F^{d \times p}} \frac{1}{n} \sum_{i=1}^n (\ip{A_i}{X X^*} - y_i^2 \log \ip{A_i}{X X^*}) + \lambda \norm{X}^2.
\end{equation}
Under a Poisson noise model for \eqref{eq:gen_model_quad}, that is, $y_i^2 \sim \poissondist(\ip{A_i}{Z_*})$,
the problem \eqref{eq:opt_ncvx_poiss} is a regularized, nonconvex, and possibly overparametrized maximum likelihood problem.
For this loss, we have the following landscape result:
\begin{theorem}
	\label{thm:landscape_poiss}
	Let the search rank $p$ satisfy
	\begin{itemize}
		\item $p \geq 3$ in the real case $\F = \R$, or
		\item $p \geq 2$ in the complex case $\F = \C$.
	\end{itemize}
	Under the model \eqref{eq:gen_model_quad} for $X_* \in \F^{d \times r}$, every second-order critical point $X$ of \eqref{eq:opt_ncvx_poiss} satisfies, for any $R \in \F^{p \times r}$,
	\begin{align*}
		&\frac{1}{n} \sum_{i=1}^n \parens*{1 - \frac{y_i^2}{\ip{A_i}{X X^*}}} \ip{A_i}{X X^* - X_* X_*^*} \\
		&\quad \leq \lambda (\norm{X_*}^2 - \norm{X}^2) \\
		&\qquad + \frac{2}{\constF p - 2} \parens*{\frac{1}{n} \norm{\alpha(X_* - X R)}^2 + \lambda \norm{X_* - X R}^2},
	\end{align*}
	where $\constF$ is defined in \Cref{lem:landscape_gen}.
\end{theorem}
This bound is less clean and developed than \Cref{thm:landscape_quad} or \Cref{thm:landscape_sqrt}.
There are many interesting ways we could develop it to try to obtain statistical landscape results (for example, the terms in the left-hand side sum resemble or can be lower bounded by various measures of distance between Poisson distributions), but we leave that to future work.

\begin{proof}[Proof of \Cref{thm:landscape_poiss}]
	For this choice of loss function $\ell$, for $b > 0$,
	\[
		\ell'(b, \upsilon) = 1 - \frac{\upsilon^2}{b}, \qquad \text{and} \qquad \ell''(w, \upsilon) = \frac{\upsilon^2}{b^2}.
	\]
	We now have the identity $b \ell''(b, \upsilon) = 1 - \ell'(b, \upsilon)$.
	Similar arguments to those in the proof of \Cref{thm:landscape_sqrt} (again based on \Cref{lem:landscape_gen}) then give, for any $Z' \succeq 0$.
	\begin{align*}
		0 &\leq (\constF p - 2) \ip{\nabla L_{\lambda}(X X^*)}{Z'} + \frac{2}{n} \ip*{ \sum_{i=1}^n A_i }{Z'} + 2 \lambda \tr Z'.
	\end{align*}
	Similarly to the proof of \Cref{thm:landscape_sqrt},
	we choose $Z' = (X_* - X R)(X_* - XR)^*$,
	and we have, again using the fact that $\nabla L_{\lambda}(X X^*) X = 0$ and handling the case $\ip{A_i}{X X^*} = 0$ similarly,
	\begin{align*}
		&\ip{\nabla L_{\lambda}(X X^*)}{Z'} \\
		&\quad= \ip{\nabla L_{\lambda}(X X^*)}{X_* X_*^* - X X^*} \\
		&\quad= \frac{1}{n} \sum_{i=1}^n \parens*{1 - \frac{y_i^2}{\ip{A_i}{X X^*}} } \ip{A_i}{X_* X_*^* - X X^*} \\
		&\qquad + \lambda (\norm{X_*}^2 - \norm{X}^2).
	\end{align*}
	The result again follows by some rearrangement.
\end{proof}

\subsection{Proof of general landscape result}
\label{sec:determ_proof}
In this section, we provide a proof of the general result \Cref{lem:landscape_gen}.
First, we need the following characterization of second-order critical points,
which is a generalization of standard gradient and Hessian calculations for factored matrix optimization (see, for example, \cite[Sec.~3.3]{Li2019}).
\begin{lemma}
	\label{lem:socp_char}
	Under the conditions of \Cref{lem:landscape_gen},
	for every second-order critical point $X$ of \eqref{eq:ncvx_gen},
	$L_\lambda(X X^*)$ is finite, and
	$\ell'(\ip{A_i}{X X^*}, y_i)$ is finite for all $i$.
	Thus $\nabla L_\lambda(X X^*)$ is well-defined.
	Furthermore,
	\begin{equation}
		\label{eq:gradcond}
		\nabla L_\lambda(X X^*) X = 0,
	\end{equation}
	and, for all $\Xdt \in \F^{d \times p}$,
	\begin{equation}
		\label{eq:hesscond}
		\begin{aligned}
			0 &\leq \ip{\nabla L_\lambda(X X^*)}{\Xdt \Xdt^*} \\
			&\quad + \frac{1}{2n} \cdot \sum_{\substack{i=1\\ \mathclap{\ip{A_i}{X X^*} > 0}}}^n \ell''(\ip{A_i}{X X^*}, y_i) \ip{A_i}{X \Xdt^* + \Xdt X^*}^2.
		\end{aligned}
	\end{equation}
\end{lemma}
\begin{proof}
	In this proof, we will use standard big-O and little-o notation:
	for functions $f(t) \geq 0$ and $g(t)$,
	we write $g(t) = O(f(t))$ to mean that $\abs{g(t)} \leq C f(t)$ for some constant $C > 0$ and all sufficiently small $t \in \R$,
	and we write $g(t) = o(f(t))$ to mean that $\lim_{t \to 0} \frac{\abs{g(t)}}{f(t)} = 0$.
	
	First, we show that $L_\lambda(X X^*)$ is finite.
	Convexity rules out the possibility that $L_\lambda(X X^*) = - \infty$.
	If $L_\lambda(X X^*) = + \infty$,
	the only way the second-order criticality condition could be satisfied (with some convention on adding infinities)
	is if, for all $\Xtl$ in some neighborhood of $X$,
	$L_\lambda(\Xtl \Xtl^*) = + \infty$.
	However, as we have assumed that each $A_i \neq 0$,
	there are matrices $\Xtl$ arbitrarily close to $X$ such that $\ip{A_i}{\Xtl \Xtl^*} > 0$ for all $i$
	(consider, for example, $\Xtl = X + t \Xdt$ for $t \neq 0$, where $\Xdt$ is random and uniformly distributed on the unit sphere; one can easily check that, with probability 1, $\ip{A_i}{\Xtl \Xtl^*} > 0$ for all $i$).
	For such $\Xtl$, $L_\lambda(\Xtl \Xtl^*)$ is finite.
	This contradicts second-order criticality, so we conclude that $L_\lambda(X X^*)$ is finite.
	
	We now consider the derivatives.
	To alleviate notation, denote $\ell'_i \coloneq \ell'(\ip{A_i}{X X^*}, y_i)$ for all $i$,
	and	$\ell''_i \coloneq \ell''(\ip{A_i}{X X^*}, y_i)$ for $i$ such that $\ip{A_i}{X X^*} > 0$.
	We furthermore denote
	\begin{align*}
		I &\coloneqq \{ i : \ip{A_i}{X X^*} = 0 \}, \qquad \text{and} \\
		J &\coloneqq \{ i : \ell'_i = - \infty \}.
	\end{align*}
	Note that, by our assumptions on $\ell$,
	we have $J \subseteq I$.
	
	Furthermore, note that, as $A_i \succeq 0$,
	\[
		\ip{A_i}{X X^*} = 0 \quad \Longleftrightarrow \quad A_i X = 0.
	\]	
	For all $i \notin I$, $\ell(b, y_i)$ is twice differentiable at $b = \ip{A_i}{X X^*} > 0$.
	Using a Taylor expansion on these terms,
	for any unit-norm $\Xdt \in \F^{d \times p}$ and sufficiently small $t \in \R$, we have
	\begin{align*}
		& L_\lambda( (X + t\Xdt)(X + t \Xdt)^* ) - L_\lambda(X X^*) \\
		&\quad= L_\lambda( X X^* + t(X \Xdt^* + \Xdt X^*) + t^2 \Xdt \Xdt^* ) - L_\lambda(X X^*) \\
		&\quad= \frac{1}{n} \sum_{i=1}^n ( \ell( \ip{A_i}{X X^* + t (X \Xdt^* + \Xdt X^*) + t^2\Xdt \Xdt^*}, y_i ) \\
		&\qquad - \ell(\ip{A_i}{X X^*}, y_i) ) + \lambda \tr( t (X \Xdt^* + \Xdt X^*) + t^2\Xdt \Xdt^* ) \\
		&\quad= \frac{1}{n} \sum_{i \in I} [\ell( t^2 \ip{A_i}{\Xdt \Xdt^*}, y_i ) - \ell(0, y_i)] \\
		&\qquad + t \ip*{ \frac{1}{n} \sum_{i \notin I} \ell'_i A_i + \lambda I_d} {X \Xdt^* + \Xdt X^*} \\
		&\qquad + t^2 \parens*{\frac{1}{n} \sum_{i \notin I} \brackets*{\ell'_i \ip{A_i}{\Xdt \Xdt^*} + \frac{\ell''_i}{2} \ip{A_i}{X \Xdt^* + \Xdt X^*}^2 } + \lambda \norm{\Xdt}^2 } \\
		&\qquad + o(t^2).
	\end{align*}
	Second-order criticality of $X$ means $L_\lambda( (X + t\Xdt)(X + t \Xdt)^* ) - L_\lambda(X X^*) \geq o(t^2)$.
	Furthermore, for all $i \in I \setminus J$,
	\begin{equation}
		\label{eq:lprime_t2}
		\begin{aligned}
		\ell( t^2 \ip{A_i}{\Xdt \Xdt^*}, y_i ) - \ell(0, y_i)
		&= t^2 \ell'_i \ip{A_i}{\Xdt \Xdt^*} + o(t^2) \\
		&= O(t^2).
		\end{aligned}
	\end{equation}
	Together with the previous Taylor expansion, these facts imply
	\begin{equation}
		\label{eq:ineq_exp_t}
		\begin{aligned}
			&\frac{1}{n} \sum_{i \in J} [\ell( t^2 \ip{A_i}{\Xdt \Xdt^*}, y_i ) - \ell(0, y_i)] \\
			&\qquad + t \ip*{ \frac{1}{n} \sum_{i \notin I} \ell'_i A_i + \lambda I_d} {X \Xdt^* + \Xdt X^*} \\
			&\quad\geq O(t^2).
		\end{aligned}
	\end{equation}
	Now, the fact that $\ell'(0, y_i) = - \infty$ for all $i \in J$ implies that, for sufficiently small $t$,
	$\ell( t^2 \ip{A_i}{\Xdt \Xdt^*}, y_i ) \leq \ell(0, y_i)$ for all $i \in J$.
	We then have, for all sufficiently small $t \in \R$,
	\[
		t \ip*{ \frac{1}{n} \sum_{i \notin I} \ell'_i A_i + \lambda I_d} {X \Xdt^* + \Xdt X^*} \geq O(t^2).
	\]
	As $t$ can be positive or negative, this inequality requires
	\begin{equation}
		\label{eq:t_term_zero}
		\ip*{ \frac{1}{n} \sum_{i \notin I} \ell'_i A_i + \lambda I_d} {X \Xdt^* + \Xdt X^*} = 0.
	\end{equation}
	We now show that $J = \varnothing$.
	Suppose, by way of contradiction, that there exists some $i \in J$.
	There is a corresponding term on the left-hand side of \eqref{eq:ineq_exp_t};
	any terms corresponding to other elements of $J$ will be nonpositive for sufficiently small $t$,
	so we obtain
	\[
		\ell( t^2 \ip{A_i}{\Xdt \Xdt^*}, y_i ) - \ell(0, y_i) \geq O(t^2).
	\]
	As $A_i \succeq 0$ and $A_i \neq 0$,
	we can choose unit-norm $\Xdt \in \F^{d \times p}$ such that $\gamma \coloneqq \ip{A_i}{\Xdt \Xdt^*} > 0$.
	Then, for some $C \geq 0$ and all sufficiently small $s > 0$, we have
	\[
		\ell( \gamma s, y_i ) - \ell(0, y_i) \geq - C s.
	\]
	This is, however, incompatible with the assumption that $\lim_{b \to 0^+}~\ell'(b, y_i) = - \infty$.
	Thus we obtain a contradiction.
	Hence, we conclude that $J = \varnothing$.
	This proves the claim that $\ell'_i$ is finite for all $i$,
	so $\nabla L_\lambda(X X^*)$ is well-defined.
	
	As $A_i X = 0$ for all $i \in I$, the equality \eqref{eq:t_term_zero} implies
	\begin{align*}
		0 &= \ip*{ \frac{1}{n} \sum_{i = 1}^n \ell'_i A_i + \lambda I_d} {X \Xdt^* + \Xdt X^*} \\
		&= \ip{\nabla L_\lambda (X X^*)}{X \Xdt^* + \Xdt X^* } \\
		&= 2 \real \ip{ \nabla L_\lambda (X X^*) X }{\Xdt}.
	\end{align*}
	As $\Xdt$ can be any unit-norm matrix in $\F^{d \times p}$, this implies \eqref{eq:gradcond}.
	
	Finally, because $J = \varnothing$, \eqref{eq:lprime_t2} holds for all $i \in I$.
	With this and \eqref{eq:t_term_zero},
	we can then further develop the Taylor expansion and second-order criticality condition as
	\begin{align*}
		o(t^2)
		&\leq L_\lambda( (X + t\Xdt)(X + t \Xdt)^* ) - L_\lambda(X X^*) \\
		&= t^2 \left( \frac{1}{n} \sum_{i=1}^n \ell'_i \ip{A_i}{\Xdt \Xdt^*} \right. \\
		&\quad + \left. \frac{1}{2n} \sum_{i \notin I} \ell''_i \ip{A_i}{X \Xdt^* + \Xdt X^*}^2 + \lambda \norm{\Xdt}^2 \right) \\
		&= t^2 \Bigg( \ip{\nabla L_\lambda(X X^*) }{\Xdt \Xdt^*} \\
		&\quad + \frac{1}{2n} \sum_{i \notin I} \ell''_i \ip{A_i}{X \Xdt^* + \Xdt X^*}^2 \Bigg).
	\end{align*}
	This implies \eqref{eq:hesscond}.
\end{proof}

%\begin{lemma}
%	\label{lem:socp_char}
%	$X \in \F^{d \times p}$ is a first-order critical point of \eqref{eq:ncvx_gen} if and only if
%	\begin{equation}
%		\label{eq:gradcond}
%		\nabla L_\lambda(X X^*) X = 0.
%	\end{equation}
%	A first-order critical point $X$ is furthermore second-order critical if and only if, for all $\Xdt \in \F^{d \times p}$,
%	\begin{equation}
%		\label{eq:hesscond}
%		0 \leq \ip{\nabla L_\lambda(X X^*)}{\Xdt \Xdt^*}
%		+ \frac{1}{2n} \sum_{i=1}^n \ell''(\ip{A_i}{X X^*}) \ip{A_i}{X \Xdt^* + \Xdt X^*}^2.
%	\end{equation}
%\end{lemma}

With this, we can prove the main landscape result:
\begin{proof}[Proof of \Cref{lem:landscape_gen}]
	Let $X$ be a second-order critical point of \eqref{eq:ncvx_gen}.
	Denote, as in the proof of \Cref{lem:socp_char},
	\[
		I = \{ i : \ip{A_i}{X X^*} = 0 \}.
	\]
	First, we use the inequality \eqref{eq:hesscond} from \Cref{lem:socp_char} with the choice of rank-1 $\Xdt = u v^*$ for some $u \in \F^d$, $v \in \F^p$.
	This gives
	\begin{equation}
		\label{eq:rank1_hess}
		\begin{aligned}
			0 &\leq \norm{v}^2 \ip{\nabla L_\lambda(X X^*)}{u u^*} \\
			&\quad + \frac{1}{2n} \sum_{i \notin I} \ell''(\ip{A_i}{X X^*}, y_i) \ip{A_i}{X v u^* + u(X v)^*}^2.
		\end{aligned}
	\end{equation}
	As noted in \cite[Sec.~3]{McRae2026a}, we have, because $A_i \succeq 0$,
	\[
		\ip{A_i}{X v u^* + u (X v)^*}^2
		\leq 4 \ip{A_i}{X v v^* X^*} \ip{A_i}{u u^*}.
	\]
	Plugging this into \eqref{eq:rank1_hess} and summing over $v$ in an orthonormal basis of $\F^p$ gives
	\begin{equation}
		\label{eq:ineq_real}
		\begin{aligned}
			0 &\leq p \ip{\nabla L_\lambda(X X^*)}{u u^*} \\
			&\quad+ \frac{2}{n} \sum_{i=1}^n \ell''(\ip{A_i}{X X^*}, y_i) \ip{A_i}{X X^*} \ip{A_i}{u u^*}.
		\end{aligned}
	\end{equation}
	We cannot, in general, do better in the real case ($\F = \R$).
	In the complex case ($\F = \C$), we can improve the constant.
	Write $A_i = B_i B_i^*$ for some matrix $B_i$.
	Then, by the fact that for $z \in \C$, $2 \real(z)^2 = \abs{z}^2 + \real(z^2)$,
	we have
	\begin{align*}
		&\ip{A_i}{X v z^* + z (X v)^*}^2 \\
		&\quad= 4 ( \real(\ip{B_i^* z}{B_i^* X v}))^2 \\
		&\quad= 2 \abs{\ip{B_i^* z}{B_i^* X v}}^2 + 2 \real( \ip{B_i^* z}{B_i^* X v}^2 ) \\
		&\quad\leq 2 \ip{A_i}{X v v^* X^*} \ip{A_i}{z z^*} + 2 \real( \ip{B_i^* z}{B_i^* X v}^2 ).
	\end{align*}
	Multiplying $v$ by the imaginary unit $i$ negates the second term in this last expression but leaves the first term unchanged.
	Then, if $\{v_k\}_{k=1}^p$ is an orthonormal basis of $\C^p$,
	plugging the above inequality into \eqref{eq:rank1_hess} and summing over all $v \in \{v_k, i v_k\}_{k=1}^p$
	gives
	\begin{equation}
		\label{eq:ineq_cplx}
		\begin{aligned}
		0 &\leq 2p \ip{ \nabla L_\lambda(X X^*)}{u u^*} \\
		&\quad+ \frac{2}{n} \sum_{i=1}^n \ell''(\ip{A_i}{X X^*}, y_i) \ip{A_i}{X X^*} \ip{A_i}{u u^*}.
		\end{aligned}
	\end{equation}
	If we write $Z' = \sum_m u_m u_m^*$,
	summing the inequality \eqref{eq:ineq_real} (in the real case) or \eqref{eq:ineq_cplx} (in the complex case) with $u = u_m$ over $m$ gives
	\begin{align*}
		0 &\leq \constF p \ip{ \nabla L_\lambda(X X^*)}{Z'} \\
		&\quad+ \frac{2}{n} \sum_{i=1}^n \ell''(\ip{A_i}{X X^*}, y_i) \ip{A_i}{X X^*} \ip{A_i}{Z'}.
	\end{align*}
	The result immediately follows.
\end{proof}

\section{Nonconvex PhaseCut and its landscape}
\label{sec:phasecut}
In this section, we develop and analyze a smooth reformulation of the nonsmooth problem \eqref{eq:opt_ncvx_sqrt} (which is a regularized version of \eqref{eq:unreg_gen_sqrt})
as a smooth \emph{synchronization} problem over spheres.
This is a nonconvex variant of the PhaseCut approach of \cite{Waldspurger2015}.
This formulation allows us to avoid any issues with nonsmoothness,
and it links our problem to the deep mathematical theory of synchronization problems.

We only consider the case of rank-1 measurements, that is, $A_i = f_i f_i^*$ for $f_1, \dots, f_n \in \F^d$.
The concepts can be extended to more general (higher-rank) PSD measurement matrices,
but the notation and analysis become far more complicated, so we omit this generalization.

Recall the notation for this case from \Cref{sec:intro}.
In the case of ordinary phase retrieval where $Z_* = x_* x_*^*$ for some vector $x_* \in \F^d$,
we have
\begin{equation*}
	y = \alpha(x_*) + \varepsilon = \abs{F x_*} + \varepsilon,
\end{equation*}
where, again,
\[
	F = \begin{bmatrix*} f_1^* \\ \vdots \\ f_n^* \end{bmatrix*}.
\]
The least-squares estimator is \eqref{eq:unreg_pr_sqrt} (which is also \eqref{eq:opt_ncvx_sqrt} with $p = 1$ and $\lambda = 0$), which we can write compactly as
\begin{equation}
	\label{eq:ls_simp}
	\min_{x \in \F^d}~\frac{1}{n} \norm{\abs{F x} - y}^2.
\end{equation}
The PhaseCut formulation comes from the observation that,
for any $a \in \F$ and $b \geq 0$,
\[
	(\abs{a} - b)^2 = \min_{s \in \F_1}~\abs{a - s b}^2,
\]
where $\F_1$ is the unit sphere in $\F$ ($\{\pm 1\}$ in $\R$ or the unit circle in $\C$).
The optimum value of $s \in \F_1$ is simply $s = \frac{a}{\abs{a}}$ if $a \neq 0$ ($s$ can be arbitrary if $a = 0$).

Thus we can rewrite the objective in the least-squares problem \eqref{eq:ls_simp} as
\begin{align*}
	\frac{1}{n} \norm{\abs{F x} - y}^2
	&= \min_{u \in \F_1^n}~\frac{1}{n} \norm{F x - \diag(y) u}^2
\end{align*}
(recall that $\diag(y) \in \R^{n \times n}$ is the diagonal matrix whose diagonal elements are the elements of $y \in \R^n$).
For fixed $u \in \F_1^n$,
a closed-form minimizer of $\norm{F x - \diag(y) u}^2$ over $x \in \F^d$ is $x = F^\dagger \diag(y) u$,
where $F^\dagger$ is the Moore-Penrose pseudoinverse of $F$.
We then have, using the fact that $F F^\dagger$ and $I_n - F F^\dagger$ are orthogonal projection matrices,
\begin{align*}
	&\min_{x \in \F^{d}}~\frac{1}{n} \norm{F x - \diag(y) u}^2 \\
	&\quad= \frac{1}{n} \norm{ F F^\dagger \diag(y) u - \diag(y) u }^2 \\
	&\quad= \frac{1}{n} \norm{(I_n - F F^\dagger) \diag(y) u}^2 \\
	&\quad= \frac{1}{n} \ip{\diag(y) (I_n - F F^\dagger) \diag(y)}{u u^*}.
\end{align*}
The resulting minimization problem over $u \in \F_1^n$ has a structure resembling max-cut (from which the name ``PhaseCut'' comes) or $\Z_2$/angular synchronization.
The paper \cite{Waldspurger2015} studies a semidefinite relaxation of this optimization problem.

We generalize this to include regularization and larger ranks of the ground truth matrices and optimization variable.
The same reformulation (without regularization) was used in the work \cite{Halsted2022} on sensor network localization, which is a particular instance of our problem.

For $X_* \in \F^{d \times \gtrank}$,
the model \eqref{eq:gen_model_sqrt} becomes
\begin{equation}
	\label{eq:phasecut_model}
	y = \abs{F X_*} + \varepsilon.
\end{equation}
We will consider a PhaseCut-like reformulation of the nonconvex problem \eqref{eq:opt_ncvx_sqrt}.
Note that, similarly to the scalar case, we have, for all $v \in \F^p$ and $b \geq 0$,
\[
	(\norm{v} - b)^2 = \min_{u \in \SFp}~\norm{v - b u}^2,
\]
where $\SFp$ is the unit sphere in $\F^p$.
For any $X \in \F^{d \times p}$,
we can apply this to each row of $F X \in \F^{n \times p}$ to obtain
\begin{align*}
	\norm{\alpha(X) - y}^2
	&= \norm{\abs{F X} - y}^2 \\
	&= \min_{U \in \F^{n \times p}}~\norm{F X - \diag(y) U}^2 \\
	&\qquad \st \diag(U U^*) = \ones.
\end{align*}
The diagonal constraint on $U U^*$ compactly represents the requirement that each row of $U$ have unit norm.

We can plug this into the problem \eqref{eq:opt_ncvx_sqrt}
and reverse the order of optimization.
For fixed $U$, the problem
\begin{equation}
	\label{eq:phasecut_ridge}
	\min_{X \in \F^{d \times p}}~\frac{1}{n} \norm{F X - \diag(y) U}^2 + \lambda \norm{X}^2
\end{equation}
is simply multivariate ridge regression.
This has the closed-form solution (unique if $\lambda > 0$ or $F$ has full row rank)
\begin{align*}
	X &= (n \lambda I_d + F^* F)^{-1} F^* \diag(y) U \\
	&= F^*(n \lambda I_n + F F^*)^{-1} \diag(y) U.
\end{align*}
If $\lambda = 0$, we can take, in a limiting sense, $X = F^\dagger \diag(y) U$ similarly to before.

As the optimal $X$ is linear in $U$,
the minimum value of \eqref{eq:phasecut_ridge} is quadratic in $U$.
We show below in \Cref{sec:phasecut_calcs} that
\begin{equation*}
	\min_{X \in \F^{d \times p}}~\frac{1}{n} \norm{F X - \diag(y) U}^2 + \lambda \norm{X}^2
	= \ip{M_\lambda}{U U^*},
\end{equation*}
where
\begin{equation*}
	M_\lambda \coloneqq \lambda \diag(y) (n \lambda I_n + F F^*)^{-1} \diag(y).
\end{equation*}

We can thus write the reformulated, constrained problem over $U$ as
\begin{equation}
	\label{eq:phasecut_gen}
	\begin{gathered}
	\min_{U \in \F^{n \times p}}~\ip{M_\lambda}{U U^*} \st \diag(U U^*) = \ones, \quad \text{where} \\
	M_\lambda = \lambda \diag(y) (n \lambda I_n + F F^*)^{-1} \diag(y).
	\end{gathered}
\end{equation}
Similarly to before, in the unregularized case $\lambda = 0$, we can take
\[
	M_0 = \lim_{\lambda \to 0^+} M_\lambda = \frac{1}{n} \diag(y) (I_n - F F^\dagger) \diag(y).
\]

We now consider the nonconvex landscape of \eqref{eq:phasecut_gen}.
The feasible set of \eqref{eq:phasecut_gen} is a Riemannian manifold (a product of spheres); apart from the (discrete) case $p = 1$ and $\F = \R$,
the problem is smooth.
A feasible point $U$ is a second-order critical point if the Riemannian gradient of the objective is zero and the Riemannian Hessian is positive semidefinite at $U$ (see \Cref{sec:phasecut_proof} for more details).
Our landscape result for such second-order critical points looks similar to \Cref{thm:landscape_sqrt}:
\begin{theorem}
	\label{thm:landscape_phasecut}
	Let the search rank $p$ satisfy
	\begin{itemize}
		\item $p \geq 2$ in the real case $\F = \R$, or
		\item $p \geq 1$ in the complex case $\F = \C$.
	\end{itemize}
	Under the model \eqref{eq:phasecut_model} for $X_* \in \F^{d \times r}$,
	every second-order critical point $U$ of \eqref{eq:phasecut_gen},
	satisfies, with $X = (n \lambda I_d + F^* F)^{-1} F^* \diag(y) U$, for all $R \in \F^{p \times \gtrank}$,
	\begin{equation}
		\label{eq:thm_phasecut_mainbd}
		\begin{aligned}
			&\frac{1}{n} \norm{\alpha(X) - \alpha(X_*)}^2 \\
			&\quad\leq \frac{2}{n} \ip{\varepsilon}{\alpha(X) - \alpha(X_*)} + \lambda (\norm{X_*}^2 - \norm{X}^2) \\
			&\qquad + \frac{1}{\constF p - 1} \parens*{ \frac{1}{n} \norm{F(X_\lambda - X R)}^2 + \lambda \norm{X_\lambda - X R}^2 },
		\end{aligned}
	\end{equation}
	where $\constF$ is defined in \Cref{lem:landscape_gen},
%	and $X_\lambda \in \F^{d \times \gtrank}$ is defined in \eqref{eq:Xlambda_def}.
	and $X_\lambda \in \F^{d \times \gtrank}$ satisfies
	\begin{equation}
		\label{eq:Xlam_errbd}
		\frac{1}{n} \norm{F(X_\lambda - X_*)}^2 + \lambda \norm{X_\lambda - X_*}^2
		\leq  \parens*{\sqrt{\lambda} \norm{x_*} + \frac{\norm{\varepsilon}}{\sqrt{n}}}^2.
	\end{equation}
\end{theorem}
The guarantee is identical to that of \Cref{thm:landscape_sqrt} apart from the presence of $X_\lambda$.
$X_\lambda$ is the solution to \eqref{eq:phasecut_ridge} given perfect knowledge of the ``true'' measurement directions; see \eqref{eq:Xlambda_def} below.
The bound on $X_\lambda - X_*$ in \eqref{eq:Xlam_errbd} is sufficient for our purposes in this paper but is not necessarily tight.

We explore some statistical consequences of \Cref{thm:landscape_phasecut} in \Cref{sec:stats}.
As discussed in \Cref{sec:phasecut_calcs} below, if $U$ were a global optimum of \eqref{eq:phasecut_gen}, we could drop the last term in \eqref{eq:thm_phasecut_mainbd} and thus obtain an identical guarantee as for global optima of \eqref{eq:opt_ncvx_sqrt} (see \Cref{thm:landscape_sqrt} and following discussion).

In the following subsections, we fill in some important details about this problem,
finishing with a full proof of \Cref{thm:landscape_phasecut}.
\subsection{Preliminary calculations}
\label{sec:phasecut_calcs}
First, we review some basic properties of ridge regression that are used above and/or will be useful for the proof of \Cref{thm:landscape_phasecut}.

The objective of the multivariate ridge regression problem \eqref{eq:phasecut_ridge} is convex and quadratic, and the global minimum satisfies
\begin{align*}
	0 &= \nabla_X \parens*{ \frac{1}{2n} \norm{F X - \diag(y) U}^2 + \frac{\lambda}{2} \norm{X}^2 } \\
	&= \frac{1}{n} F^* (F X - \diag(y) U) + \lambda X.
\end{align*}
We will write the closed-form solution as
\[
	X = \recmap U,
\]
where
\begin{align*}
	\recmap &\coloneqq (n\lambda I_d + F^* F)^{-1} F^* \diag(y) \\
	& = F^*(n\lambda I_n + F F^*)^{-1} \diag(y).
\end{align*}
Note that, by construction of $\recmap$, the zero-gradient--like condition
\[
	0 = \frac{1}{n} F^*(F \recmap W - \diag(y) W) + \lambda \recmap W
\]
holds for any $W \in \F^{n \times p}$.
This gives the identity
\begin{align*}
	&\frac{1}{n} \norm{F \recmap W - \diag(y) W}^2 + \lambda \norm{\recmap W}^2 \\
	&\quad= \frac{1}{n} ( \norm{\diag(y) W}^2 + 2 \underbrace{\ip{F \recmap W}{F \recmap W - \diag(y) W}}_{= - n\lambda \norm{\recmap W}^2} \\
	&\qquad  - \norm{F \recmap W}^2 ) + \lambda \norm{\recmap W}^2 \\
	&= \frac{1}{n} ( \norm{\diag(y) W}^2 - \norm{F \recmap W}^2) - \lambda \norm{\recmap W}^2 \\
	&= \ip{M_\lambda W}{W},
\end{align*}
where, recalling the definition of $\recmap$,
\begin{align*}
	M_\lambda
	&\coloneqq \frac{1}{n} (\diag(y)^2 - \recmap^* (F^* F + n \lambda I_n) \recmap) \\
%	&= \frac{1}{n} \diag(y)( I_n - F(n\lambda I_d + F^* F)^{-1} F^* F (n \lambda I_d + F^* F)^{-1} F^* - n\lambda F (n\lambda I_d + F^* F)^{-2} F^*  ) \diag(y) \\
%	&= \frac{1}{n}\diag(y)( I_n - F(n\lambda I_d + F^* F)^{-1} F^* ) \diag(y) \\
	&= \frac{1}{n} \diag(y)( I_n - F F^* (n\lambda I_n + F F^*)^{-1}) \diag(y) \\
	&= \lambda\diag(y) (n\lambda I_n + F F^*)^{-1} \diag(y).
\end{align*}
The intermediate identity
\begin{equation}
	\label{eq:quad_id}
	\ip{M_\lambda W}{W} = \frac{1}{n} \norm{\diag(y) W}^2 - \frac{1}{n}  \norm{F \recmap W}^2 - \lambda \norm{\recmap W}^2,
\end{equation}
which holds for all $W \in \F^{n \times p}$, will later prove useful.

Let $U_* \in \F^{n \times \gtrank}$ be a matrix of ``ground truth'' directions of $F X_*$, that is, a\footnote{The choice is not necessarily unique if $(F X_*)_i = 0$ for some $i$.} matrix with unit-norm rows satisfying
\begin{align*}
	(F X_*)_i &= \norm{(F X_*)_i}(U_*)_i \quad \forall i, \quad \text{or} \\
	F X_* &= \diag(\abs{F X_*}) U_*.
\end{align*}
The matrix $X_\lambda$ in \Cref{thm:landscape_phasecut} is the optimum of \eqref{eq:phasecut_ridge} (with $p = r$) given $U = U_*$:
\begin{equation}
	\label{eq:Xlambda_def}
	\begin{aligned}
		X_\lambda
		& \coloneqq \recmap U_* \\
		&= \frac{1}{n} \parens*{ \lambda I_d + \frac{1}{n} F^* F }^{-1} F^*( F X_* + \diag(\varepsilon) U_* ) \\
		&= X_* + \parens*{ \lambda I_d + \frac{1}{n} F^* F }^{-1} \parens*{ -\lambda X_* + \frac{1}{n} F^* \diag(\varepsilon) U_* }.
	\end{aligned}
\end{equation}
Our analysis is somewhat complicated by the fact that we do not obtain $X_*$ exactly.
The error term in this last expression is standard from (multivariate) ridge regression.
As $U_*$ is not necessarily uniquely defined, neither is $X_\lambda$.
However, this makes no difference to our analysis in this paper.

We now consider some properties of optima of \eqref{eq:phasecut_gen}.
We do this in several steps.
First, we can estimate
\begin{align*}
	\ip{M_\lambda}{U_* U_*^*}
	&= \frac{1}{n} \norm{F X_\lambda - \diag(y) U_*}^2 + \lambda \norm{X_\lambda}^2 \\
	&\leq \frac{1}{n} \norm{F X_* - \diag(y) U_*}^2 + \lambda \norm{X_*}^2 \\
	&= \frac{\norm{\varepsilon}^2}{n} + \lambda \norm{X_*}^2.
\end{align*}
This upper bound (which follows from the fact that $X_\lambda$ minimizes the previous expression over all $X \in \F^{d \times \gtrank}$)
will prove simpler to use than the exact expression.

Next, for any feasible $U$ of \eqref{eq:phasecut_gen},
for $X = \recmap U$,
we have
\begin{align*}
	\ip{M_\lambda}{U U^*}
	&= \frac{1}{n} \norm{F X - \diag(y) U}^2 + \lambda \norm{X}^2 \\
	&\geq \frac{1}{n} \norm{\abs{F X} - y}^2 + \lambda \norm{X}^2 \\
	&= \frac{1}{n} \norm{\alpha(X) - \alpha(X_*) - \varepsilon}^2 + \lambda \norm{X}^2 \\
	&= \frac{1}{n} (\norm{\varepsilon}^2 + \norm{\alpha(X) - \alpha(X_*)}^2 \\
	&\quad - 2 \ip{\varepsilon}{\alpha(X) - \alpha(X_*)}) + \lambda \norm{X}^2.
\end{align*}
Combining these previous two inequalities, we obtain
\begin{equation}
	\label{eq:phasecut_basicdiff}
	\begin{aligned}
		&\ip{M_\lambda}{U U^* - U_* U_*^*} \\
		&\quad \geq \frac{1}{n} \norm{\alpha(X) - \alpha(X_*)}^2 - \frac{2}{n} \ip{\varepsilon}{\alpha(X) - \alpha(X_*)} \\
		&\qquad + \lambda (\norm{X}^2 - \norm{X_*}^2).
	\end{aligned}
\end{equation}
However, if $U$ is globally optimal\footnote{We could find a global optimum, for example, by taking $p = n$ and using semidefinite programming; the original PhaseCut paper \cite{Waldspurger2015} only considered such semidefinite relaxations.} for \eqref{eq:phasecut_gen} (with $p \geq \gtrank$),
the left-hand-side of \eqref{eq:phasecut_basicdiff} must be $\leq 0$,
so we obtain
\[
	\frac{1}{n} \norm{\alpha(X) - \alpha(X_*)}^2
	\leq \frac{2}{n} \ip{\varepsilon}{\alpha(X) - \alpha(X_*)} + \lambda (\norm{X_*}^2 - \norm{X}^2).
\]
Comparing to \Cref{thm:landscape_sqrt},
this is precisely the same inequality we would obtain from a global optimum (or semidefinite relaxation) of \eqref{eq:opt_ncvx_sqrt}.
The additional term in \Cref{thm:landscape_phasecut} is similarly the price we pay for only having a second-order critical point of the nonconvex problem \eqref{eq:phasecut_gen}.

\subsection{Nonconvex landscape proof}
\label{sec:phasecut_proof}
Finally, we analyze the nonconvex landscape of \eqref{eq:phasecut_gen}.
As quadratically constrained quadratic problems of this particular form are by now well-studied and are not the main focus of this paper, we omit much of the context and details.
See, for example, \cite{Endor2026,McRae2025preprinta,Boumal2019} for further details and references.
The following lemma provides the properties we will need of second-order critical points of \eqref{eq:phasecut_gen}.
The operator $\ddiag$ on $n \times n$ matrices extracts the diagonal, setting off-diagonal elements to zero.
We write the Hadamard (elementwise) product between two matrices $A$ and $B$ of equal size as $A \circ B$, and we denote the Hadamard square by $A^{\circ 2} = A \circ A$.
\begin{lemma}
	\label{lem:phasecut_socp}
	Let $M \in \F^{n \times n}$ be Hermitian.
	Every second-order critical point $U$ of the problem
	\[
		\min_{U \in \F^{n \times p}}~\ip{M}{U U^*} \st \diag(U U^*) = \ones
	\]
	satisfies the following properties,
	where
	\[
		S(U) \coloneqq M - \real(\ddiag(M U U^*)).
	\]
	\begin{itemize}
		\item $S(U) U = 0$, and,
		\item For all $z \in \F^n$,
		\begin{equation}
			\label{eq:ineq_phasecut_socp}
			\ip{S(U)}{(\constF p - 2) z z^* + \real(D_z^* U U^* D_z) \circ (U U^*)} \geq 0,
		\end{equation}
		where $D_z = \diag(z)$.
	\end{itemize}
\end{lemma}
This is a slight generalization of intermediate results in \cite{McRae2025preprinta};
we provide a proof at the end of this section.
With this and the calculations developed in \Cref{sec:phasecut_calcs} above,
we can prove the main landscape result.
\begin{proof}[Proof of \Cref{thm:landscape_phasecut}]
	Note that, for $z \in \F^n$, we can rewrite the inequality \eqref{eq:ineq_phasecut_socp} from \Cref{lem:socp_char} (with $M = M_\lambda$) as
	\begin{align*}
		0 &\leq (\constF p - 1) \ip{S(U)}{z z^*} \\
		&\quad + \ip{S(U)}{\real(D_z^* U U^* D_z) \circ (U U^*) - z z^*} \\
		&= (\constF p - 1) \ip{S(U)}{z z^*} \\
		&\quad + \ip{M_\lambda}{\real(D_z^* U U^* D_z) \circ (U U^*) - z z^*} \\
		&\quad - \underbrace{\ip{\real(\ddiag(M_\lambda U U^*))}{\real(D_z^* U U^* D_z) \circ (U U^*) - z z^*}}_{= 0} \\
		&= (\constF p - 1) \ip{S(U)}{z z^*} \\
		&\quad + \ip{M_\lambda}{\real(D_z^* U U^* D_z) \circ (U U^*) - z z^*}.
	\end{align*}
	The third inner product on the middle line is zero because the matrix on the right-hand side has zero diagonal.
	
	As $M_\lambda \preceq \frac{1}{n}\diag(y)^2$,
	\begin{align*}
		&\ip{M_\lambda}{\real(D_z^* U U^* D_z) \circ (U U^*)} \\
		&\quad \leq \frac{1}{n} \tr( \diag(y) (\real(D_z^* U U^* D_z) \circ (U U^*)) \diag(y) ) \\
		&\quad= \frac{1}{n} \norm{\diag(y) z}^2.
	\end{align*}
	Furthermore, by \eqref{eq:quad_id},
	\[
		\ip{M_\lambda}{z z^*} = \frac{1}{n} (\norm{\diag(y) z}^2 - \norm{F \recmap z}^2) - \lambda \norm{\recmap z}^2.
	\]
	Combining these previous three displays, we obtain
	\[
		0 \leq (\constF p - 1) \ip{S(U)}{z z^*} + \frac{1}{n} \norm{F \recmap z}^2 + \lambda \norm{\recmap z}^2.
	\]
	Now, choose $z = z_k = (U_* - U R)v_k$,
	where $\{ v_k \}_{k=1}^p$ is an orthonormal basis for $\F^p$,
	and sum up the resulting inequalities to obtain
	\begin{align*}
		0 &\leq (\constF p - 1) \ip{S(U)}{(U_* - U R)(U_* - U R)^*} \\
		&\quad + \frac{1}{n} \norm{F \recmap (U_* - U R)}^2 + \lambda \norm{\recmap(U_* - U R)}^2 \\
		&= (\constF p - 1) \ip{S(U)}{(U_* - U R)(U_* - U R)^*} \\
		&\quad + \frac{1}{n} \norm{F (X_\lambda - X R)}^2 + \lambda \norm{X_\lambda - X R}^2.
	\end{align*}
	The fact that $S(U) U = 0$ from \Cref{lem:socp_char} implies
	\begin{align*}
		\ip{S(U)}{(U_* - U R)(U_* - U R)^*} &= \ip{S(U)}{U_* U_*^* - U U^*} \\
		&= \ip{M_\lambda}{U_* U_*^* - U U^*}.
	\end{align*}
	
	Thus we obtain
	\begin{align*}
		&\ip{M_\lambda}{U U^* - U_* U_*^*} \\
		&\quad\leq \frac{1}{\constF p - 1} \parens*{ \frac{1}{n} \norm{F(X_\lambda - X R)}^2 + \lambda \norm{X_\lambda - X R}^2}.
	\end{align*}
	Together with \eqref{eq:phasecut_basicdiff}, this proves the main inequality \eqref{eq:thm_phasecut_mainbd}.
	
	Finally, using \eqref{eq:Xlambda_def}, we can bound
	\begin{align*}
		&\parens*{\frac{1}{n} \norm{F (X_\lambda - X_*)}^2 + \lambda \norm{X_\lambda - X_*}^2}^{1/2} \\
		&\quad= \ip*{ \parens*{ \frac{1}{n} F^* F + \lambda I_d } }{(X_\lambda - X_*)(X_\lambda - X_*)^*}^{1/2} \\
		&\quad= \left\langle \parens*{ \lambda I_d + \frac{1}{n} F^* F }^{-1}, \right. \\
		&\qquad \left. \parens*{ -\lambda X_* + \frac{1}{n} F^* \diag(\varepsilon) U_* } \parens*{ -\lambda X_* + \frac{1}{n} F^* \diag(\varepsilon) U_* }^* \right\rangle^{1/2} \\
		&\quad\leq \ip*{ \parens*{ \lambda I_d + \frac{1}{n} F^* F }^{-1} }{\lambda^2 X_* X_*^* }^{1/2} \\
		&\qquad + \ip*{ \frac{1}{n} F \parens*{ \lambda I_d + \frac{1}{n} F^* F }^{-1} F^* }{ \frac{1}{n} (\diag(\varepsilon) U_* )( \diag(\varepsilon) U_* )^* }^{1/2} \\
		&\leq ( \lambda \tr(X_* X_*^*) )^{1/2} + \parens*{ \frac{1}{n} \tr (\diag(\varepsilon) U_* U_*^* \diag(\varepsilon)) }^{1/2} \\
		&\quad= \sqrt{\lambda} \norm{X_*} + \frac{\norm{\varepsilon}}{\sqrt{n}},
	\end{align*}
	which gives \eqref{eq:Xlam_errbd}.
\end{proof}

To finish, we sketch the proof of our auxiliary landscape lemma based on arguments from \cite{McRae2025preprinta,Endor2026}.
\begin{proof}[Proof of \Cref{lem:phasecut_socp}]
	The optimization problem is that of a smooth function on a smooth Riemannian manifold;
	in particular, the constraint set is a product of $n$ unit spheres in $\F^p$.
	First-order criticality of a feasible point $U$ is equivalent to $S(U) U = 0$,
	as this quantity is proportional to the Riemannian gradient of the objective at $U$.
	
	Second-order criticality requires, in addition to first-order criticality, that
	the Riemannian Hessian be positive semidefinite at $U$.
	For this problem, that means
	\[
		\ip{S(U)}{\Udt \Udt^*} \geq 0
	\]
	for all
	\[
		\Udt \in \tansp \coloneqq \{ \Udt \in \F^{n \times p} : \diag(U \Udt^* + \Udt U^*) = 0 \}.
	\]
	$\tansp$ is the tangent space to the constraint manifold at $U$.
	We denote by $\Ptansp$ the orthogonal projection (in $\F^{n \times p}$) to $\tansp$.
	To obtain a useful inequality,
	we consider, for unit-norm $v \in \F^p$, the tangent vector
	\begin{align*}
		\Udt_v &= \Ptansp(z v^*) \\
		&= z v^* - \ddiag(\real(z v^* U^*)) U \\
		&= z v^* - \real(D_z^* \diag(U v)) U.
	\end{align*}
	Then
	\begin{align*}
		\Udt_v \Udt_v^* &= z z^* - \real(D_z^* \diag(U v)) U v z^* \\
		&\quad - z (Uv)^* \real(D_z^* \diag(U v)) \\
		&\quad + \real(D_z^* \diag(U v)) U U^* \real(D_z^* \diag(U v)).
	\end{align*}
	We now consider separately the real and complex cases.
	In the real case ($\F = \R$),
	we have
	\begin{align*}
		\Udt_v \Udt_v^* &= z z^* - D_z \abs{U v}^2 z^* - z (\abs{Uv}^2)^* D_z \\
		&\qquad + D_z \diag(U v) U U^* \diag(U v) D_z,
	\end{align*}
	where, as before, $\abs{Uv}^2$ is the elementwise squared absolute value of the vector $Uv$.
	Note that, if $\{ v_k \}_{k=1}^p$ is an orthonormal basis for $\R^p$,
	we have, for all $i,j \in \{1, \dots, n\}$,
	\[
		\sum_{k=1}^p (U v_k)_i (U v_k)_j = (U U^*)_{ij}.
	\]
	Therefore,
	\[
		\sum_{k=1}^p D_z \abs{U v_k}^2 z^* = D_z \ones z^* = z z^*,
	\]
	and
	\[
		\sum_{k=1}^p D_z \diag(U v_k) U U^* \diag(U v_k) D_z
		= D_z (U U^*)^{\circ 2} D_z.
	\]
	This implies
	\[
		\sum_{k=1}^p \Udt_{v_k} \Udt_{v_k}^* = (p - 2) z z^* + D_z (U U^*)^{\circ 2} D_z,
	\]
	so, finally,
	\begin{align*}
		0 &\leq \sum_{k=1}^p \ip{S(U)}{\Udt_{v_k} \Udt_{v_k}^*} \\
		&= \ip*{ S(U) }{ \sum_{k=1}^p \Udt_{v_k} \Udt_{v_k}^* } \\
		&= \ip{S(U)}{(p-2)z z^* + D_z (U U^*)^{\circ 2} D_z}.
	\end{align*}
	This completes the proof in the real case.
	
	In the complex case ($\F = \C$),
	the argument is similar but messier.
	We now have
	\begin{align*}
		\Udt_v \Udt_v^*
		&= z z^*  - \frac{1}{2} (D_z^* \diag(Uv) + D_z \diag(Uv)^*) Uv z^* \\
		&\quad - \frac{1}{2} z (Uv)^* (\diag(Uv)^* D_z + \diag(Uv) D_z^*) \\
		&\quad + \frac{1}{4} (D_z^* \diag(Uv) + D_z \diag(Uv)^*) U U^* \\
		&\qquad \times (\diag(Uv)^* D_z + \diag(Uv) D_z^*).
	\end{align*}
	To simplify this, first, one can easily verify that
	\begin{align*}
		\Udt_v \Udt_v^*
		+ \Udt_{i v} \Udt_{i v}^*
		&= 2 z z^* - D_z \abs{Uv}^2 z^* - z (\abs{Uv}^2)^* D_z^* \\
		&\quad + \frac{1}{2} D_z^* \diag(Uv) U U^* \diag(Uv)^* D_z \\
		&\quad + \frac{1}{2} D_z \diag(Uv)^* U U^* \diag(Uv) D_z^*.
	\end{align*}
	If $\{ v_k \}_{k=1}^p$ is an orthonormal basis of $\C^p$,
	we have
	\[
		\sum_{k=1}^p (U v_k)_i (U v_k)_j^* = (U U^*)_{ij},
	\]
	so
	\begin{align*}
		&\sum_{k=1}^p (\Udt_{v_k} \Udt_{v_k}^* + \Udt_{iv_k} \Udt_{iv_k}^*) \\
		&\quad= (2p - 2) z z^* + \frac{1}{2} ( D_z^* (U U^*)^{\circ 2} D_z + D_z \abs{U U^*}^2 D_z^* ) \\
		&\quad= (2p - 2) z z^* + \real(D_z^* U U^* D_z) \circ (U U^*).
	\end{align*}
	The result follows similarly to the real case.
\end{proof}
\section{Some statistical consequences for phase retrieval}
\label{sec:stats}
In this section, we give some statistical results for ordinary phase retrieval, that is,
recovering a vector $x_*$ from (noisy) measurements of the form $y_i \approx \abs{\ip{f_i}{x_*}}$.
The results are not exhaustive but are intended to illustrate the usefulness of the deterministic tools we have developed in \Cref{sec:determ_direct,sec:phasecut}.
Further implications, in particular results for general rank-$r$ semidefinite matrix sensing problems, are left to future work.

\subsection{Finite dimension, isotropic sub-Gaussian measurements}
\label{sec:subG}
First, we consider the case of estimating $x_* \in \F^d$ from measurements of the form $y_i \approx \abs{\ip{f_i}{x_*}}$,
where $f_1, \dots, f_n$ are isotropic and sub-Gaussian.
We adopt the assumptions of \cite{Krahmer2018,Krahmer2020},
presented as stated in \cite{McRae2026a}; see those papers for further context and references.
These assumptions are, to the best of our knowledge, among the most general under which one can obtain known-optimal statistical guarantees.
Specifically, we assume that the measurement vectors $f_i$ are independent and identically distributed (i.i.d.),
and, furthermore, that their entries are i.i.d.\ copies of a sub-Gaussian zero-mean random variable $\crdvar$.
We say\footnote{This is one of several equivalent conditions common in the literature.} that a zero-mean random variable $\crdvar$ is $K$--sub-Gaussian for $K > 0$ if $\E e^{\abs{w}^2/K^2} \leq 2$.
We additionally need some moment conditions on $w$ to ensure that the phase retrieval map $x_* \mapsto \alpha(x_*)$ is injective (modulo trivial symmetries).

The following statistical landscape result for the nonconvex problems \eqref{eq:unreg_gen_sqrt} (equivalently, \eqref{eq:opt_ncvx_sqrt} with $\lambda = 0$) and \eqref{eq:phasecut_gen} is a counterpart to \cite[Thm.~3]{McRae2025preprinta}, which instead studied the landscape of \eqref{eq:unreg_gen_quartic}:
\begin{theorem}
	\label{thm:finite_subG}
	Consider the model \eqref{eq:gen_model_sqrt} with rank-one $Z_* = x_* x_*^*$ for nonzero $x_* \in \F^d$.
	Suppose $A_i = f_i f_i^*$, where $f_1, \dots, f_n$ are i.i.d.\ random vectors whose entries are i.i.d.\ copies of a zero-mean random variable $\crdvar$.
	If $\F = \R$ but $w$ is complex, we can take $A_i = \real(f_i f_i^*)$.
	Suppose the following properties are true about $\crdvar$ and $x_*$:
	\begin{itemize}
		\item $\E \abs{\crdvar}^2 = 1$.
		\item $\crdvar$ is $K$--sub-Gaussian for some $K > 0$.
		\item At least one of the following two statements holds:
		\begin{enumerate}
			\item $\E \abs{\crdvar}^4 > 1$, or
			\item $\norm{x_*}_\infty \leq \incohp \norm{x_*}$ for a sufficiently small universal constant $\incohp > 0$.
		\end{enumerate}
		\item If $\F = \C$, $\abs{\E \crdvar^2} < 1$.
	\end{itemize}
	Then there exist $c_1, c_2, c_3, c_4, c_5 > 0$ depending only on the properties of $\crdvar$ (not on the dimension $d$) such that, if $n \geq c_1 d$, with probability at least $1 - c_2 n^{-2}$,
	for all
	\[
		p \geq c_3,
	\]
	for every second-order critical point $X$ of \eqref{eq:unreg_gen_sqrt} or (in the case $A_i = f_i f_i^*$) $X = F^\dagger \diag(y) U$ where $U$ is any second-order critical point of \eqref{eq:phasecut_gen} with $\lambda = 0$, we have
	\[
		\nucnorm{X X^* - x_* x_*^*} \leq c_4 \parens*{ \norm{x_*} \frac{\norm{\varepsilon}}{\sqrt{n}} + \frac{\norm{\varepsilon}^2}{n} }.
	\]
	Furthermore, if a closest rank-1 approximation to $X$ in $\norm{\cdot}$ is $\xhat v^*$ 
	for some $\xhat \in \F^d$ and unit-norm $v \in \F^p$, we have
	\[
		\min_{\abs{s} = 1}~\norm{\xhat - s x_*}
		\leq c_5 \frac{\norm{\varepsilon}}{\sqrt{n}}.
	\]
\end{theorem}
We prove this in \Cref{sec:stats_proofs}.
A version of this result for more general measurement covariance and $\lambda \geq 0$ (but assuming Gaussian measurement vectors) appears as \Cref{thm:infdim} below.
In contrast to \cite[Thm.~3]{McRae2025preprinta}, which requires a logarithmic overparametrization $p \approx d \log d$ in the case $n \approx d$,
\Cref{thm:finite_subG} only requires overparametrization by the constant factor $c_3$.
On the other hand, unlike certain more specialized results in \cite{McRae2026a},
\Cref{thm:finite_subG} does not apply when $p = 1$.
It is not clear under what conditions we might expect the exactly-parametrized problem \eqref{eq:unreg_pr_sqrt} to have a benign landscape,
and a tighter analysis (in terms of constants) would be necessary to prove such a result.

The error bounds of \Cref{thm:finite_subG} with respect to noise are optimal within constants.
Indeed, if $\varepsilon = \delta \alpha(x_*)$ for some $\delta \geq 0$,
we have, with high probability, $\norm{\varepsilon} \approx \delta \sqrt{n} \norm{x_*}$,
and, as $y = (1 + \delta) \alpha(x_*)$, any exact recovery algorithm will return $x_\delta \coloneqq (1 + \delta) x_*$,
which has vector error
\[
	\norm{x_\delta - x_*}
	= \delta \norm{x_*}
	\approx \frac{\norm{\varepsilon}}{\sqrt{n}}
\]
and matrix error
\begin{align*}
	\nucnorm{x_\delta x_\delta^* - x_* x_*^*}
	&= [(1 + \delta)^2 - 1] \norm{x_*}^2 \\
	&= (2 \delta + \delta^2) \norm{x_*}^2 \\
	&\approx \norm{x_*} \frac{\norm{\varepsilon}}{\sqrt{n}} + \frac{\norm{\varepsilon}^2}{n}.
\end{align*}

Compared to \cite[Thm.~3]{McRae2025preprinta}, which considers the problem \eqref{eq:unreg_gen_quartic},
we only need the optimization rank $p$ to be constant rather than (potentially) logarithmic in the dimension $d$. On the other hand, if the noise is non-adversarial (e.g., i.i.d.\ Gaussian), \Cref{thm:finite_subG} gives a poorer error bound scaling (by a factor of $\sqrt{\frac{n}{d}}$ for large $n$) than \cite[Thm.~3]{McRae2025preprinta}.
We believe this is an artifact of our analysis, as initial numerical experiments suggest that \eqref{eq:unreg_gen_sqrt} has similar or better statistical performance with noise than \eqref{eq:unreg_gen_quartic}, but it is unclear how to improve our analysis.
The biggest technical obstacle seems to be the fact, noted in \cite{Balan2022}, that for $p > 1$ we cannot in general lower bound $\norm{\alpha(X) - \alpha(x_*)}$ by some constant multiple of $\min_{\norm{v} = 1}~\norm{X - x_* v^*}$.

\subsection{Infinite dimension, Gaussian measurements}
\label{sec:infdim}
Next, we consider the case where the dimension of the problem is potentially large (even infinite) compared to the number of measurements $n$.
%In this case, we must suitable adjust our expectations for how well we can recover the ground truth $x_*$.
We will assume, for simplicity, that the measurements are Gaussian, though this can likely be relaxed;
an interesting potential further development would be recovery results for phase retrieval in a general reproducing kernel Hilbert space.

Let $\scrH$ be a (real or complex) Hilbert space with inner product $\ip{\cdot}{\cdot}$ and induced norm $\norm{\cdot}$.
We say $f \in \scrH$ is a zero-mean Gaussian random variable on $\scrH$ with covariance $\Sigma$ if,
for every $u \in \scrH$, $\ip{u}{f}$ is zero-mean Gaussian,
and, for every $u, v \in \scrH$, $\E \ip{u}{f} \overline{\ip{v}{f}} = \ip{\Sigma u}{v}$.
We assume $\Sigma$ is \emph{trace-class}, that is, $\tr \Sigma = \E \norm{f}^2 < \infty$.

We denote, for $u, v \in \scrH$, the $L_2$ inner product
\begin{align*}
	\iplt{u}{v} \coloneqq \E \ip{u}{f} \overline{\ip{v}{f}} = \ip{\Sigma u}{v}
\end{align*}
with induced norm
\[
	\normlt{u}^2 \coloneqq \E \abs{\ip{u}{f}}^2 = \norm{\Sigma^{1/2} u}^2.
\]
With this, we can define various norms on operators on $\scrH$.
In particular, we will use the $L_2$ nuclear and operator norms
\begin{align*}
	\nucnormlt{T} &\coloneqq \nucnorm{\Sigma^{1/2} T \Sigma^{1/2}} \qquad  \text{and} \\
	\opnormlt{T} &\coloneqq \opnorm{\Sigma^{1/2} T \Sigma^{1/2}}
\end{align*}
defined for trace-class operators $T$ on $\scrH$,
where $\opnorm{\cdot}$ and $\nucnorm{\cdot}$ are the operator and nuclear norms as defined on operators on $\scrH$.

If $\scrH$ is infinite-dimensional, we cannot hope to recover arbitrary $x_* \in \scrH$ exactly with finite samples.
Instead, we try to estimate $x_*$ accurately in an $L_2$ or regression sense.
Thus we aim not to ``invert'' the measurement process but rather to ``learn'' the mapping $f \mapsto \ip{x_*}{f}$.
In the linear regression case,
where we observe i.i.d\ copies of $(f, \ip{x_*}{f})$ without the nonlinear absolute value,
the difficulty is determined by the eigenvalues of $\Sigma$.
Indeed, these will feature prominently in the phase retrieval result below.

To estimate $x_*$ from measurements of the form $y_i \approx \abs{\ip{f_i}{x_*}}$,
we can still use the algorithms and theory of \Cref{sec:determ_direct,sec:phasecut}.
However, we must make several adaptations and clarifications:
\begin{itemize}
	\item We replace $\F^{d \times p}$ by $\scrH^p$.
	We overload our previous notation to write, for $X = (x_1, \dots, x_p) \in \scrH^p$,
	\[
		\norm{X}^2 = \sum_{k=1}^p \norm{x_k}^2, \qquad \normlt{X}^2 = \sum_{k=1}^p \normlt{x_k}^2,
	\]
	and
	\[
		X X^* = \sum_{k=1}^p x_k^{\otimes 2},
	\]
	where $x^{\otimes 2} = x \otimes x$ denotes the tensor product of $x \in \scrH$ with itself.
	Furthermore, for any linear operator $T$ with domain $\scrH$, we write $T X = (T x_1, \dots, T x_n)$.
	\item For measurement vectors $f_1, \dots, f_n \in \scrH$, we now denote by $F \colon \scrH \to \F^n$ the linear operator defined by
	\[
		F x = \begin{bmatrix*}
			\ip{x}{f_1} \\ \vdots \\ \ip{x}{f_n}.
		\end{bmatrix*}.
	\]
	With the linear map notation overloading from the previous bullet point, we can then write, for $X \in \scrH^p$,
	\[
		\alpha(X) = \abs{F X}
	\]
	similarly to the finite-dimensional case.
	\item With this notation, the nonconvex problem \eqref{eq:opt_ncvx_sqrt} becomes
	\begin{equation}
		\label{eq:infdim_opt_sqrt}
		\min_{X \in \scrH^p}~\frac{1}{n} \norm{\alpha(X) - y}^2 + \lambda \norm{X}^2,
	\end{equation}
	while the PhaseCut-reformulated problem \eqref{eq:phasecut_gen} remains unchanged,
	noting that the matrix $F F^*$ is given by $(F F^*)_{ij} = \ip{f_j}{f_i}$.
	\item We do not consider computation in detail, but, for example, this can be done by the standard ``kernel trick'' if we have access to inner products $\ip{u}{v}$ for arbitrary $u,v \in \scrH$.
\end{itemize}

With these adaptations, we have the following nonconvex landscape and recovery result:
\newcommand{\tgtdim}{d}
\newcommand{\Sigperp}{\Sigma_\tgtdim^\perp}
\begin{theorem}
	\label{thm:infdim}
	Consider the model $y = \alpha(x_*) + \varepsilon$
	for nonzero $x_* \in \scrH$, where $\scrH$ is a Hilbert space over $\F$,
	$f_1, \dots, f_n$ are i.i.d.\ Gaussian random vectors in $\scrH$ (circularly symmetric if $\F = \C$) with trace-class covariance $\Sigma$,
	and $\varepsilon \in \R^n$ is such that $y_1, \dots, y_n \geq 0$.
	Let the eigenvalues of $\Sigma$ be
	\[
		\sigma_1 \geq \sigma_2 \geq \cdots \geq 0.
	\]
	There exist universal constants $c_1, \dots, c_6 > 0$ such that the following holds.
	
	Fix a positive integer $d \leq \rank(\Sigma)$.
	If $n \geq c_1 \tgtdim$,
	with probability at least $1 - c_2 n^{-2}$,
	for all
	\[
		p \geq c_3 \quad \text{and} \quad \lambda \geq c_4\parens*{ \sigma_{\tgtdim + 1} + \frac{1}{n} \sum_{m > \tgtdim} \sigma_m }
	\]
	(with the convention that, if $d = \rank(\Sigma) < \infty$, $\sigma_{d+1} = 0$),
	every second-order critical point $X$ of \eqref{eq:infdim_opt_sqrt} or, alternatively,
	\[
		X = F^*(n \lambda I_n + F F^*)^{-1} \diag(y) U,
	\]
	where $U$ is any second-order critical point of \eqref{eq:phasecut_gen},
	satisfies
	\[
		\nucnormlt{X X^* - x_* x_*^*}
		\leq c_5 \parens*{ \frac{\norm{\varepsilon}^2}{n} + \normlt{x_*} \parens*{\frac{\norm{\varepsilon}}{\sqrt{n}} + \sqrt{\lambda} \norm{x_*} } }.
	\]
	Furthermore, if the closest rank-1 approximation to $X$ in $\normlt{\cdot}$ is $\xhat v^* = (\vbr_1 \xhat, \dots, \vbr_p \xhat)$ 
	for some $\xhat \in \scrH$ and unit-norm $v \in \F^p$, we have
	\[
		\min_{\abs{s} = 1}~\normlt{\xhat - s x_*}
		\leq c_6 \parens*{\frac{\norm{\varepsilon}}{\sqrt{n}} + \sqrt{\lambda} \norm{x_*} }.
	\]
\end{theorem}
The assumption that each $y_i \geq 0$ was introduced and discussed in \Cref{sec:intro_ms}.

If $\rank(\Sigma) < \infty$, and we choose $d = \rank(\Sigma)$,
we recover a version (more general in terms of $\Sigma$ and $\lambda$) of \Cref{thm:finite_subG}.
Hence the sample complexity and the error terms involving $\varepsilon$ are optimal.

The error term involving $\lambda$ is also optimal within constants.
To isolate this term, consider $\varepsilon = 0$;
we then have, for some unit-magnitude $s \in \F$,
\[
	\normlt{\xhat - x_* s}
	\lesssim \sqrt{\lambda} \norm{x_*}.
\]
An error term of this form is standard even in ordinary linear Hilbert space regression,
and it is optimal in some cases.
The requirement on $\lambda$ in terms of the eigenvalues is also similar to conditions appearing in existing results for infinite-dimensional linear regression. See, for example, \cite{Hsu2014,McRae2020} for some relevant results and further reading.

\subsection{Proofs}
\label{sec:stats_proofs}
We first prove the somewhat simpler \Cref{thm:finite_subG}.
Many of the techniques and necessary intermediate results will carry over to the proof of \Cref{thm:infdim}.

First, we state a standard result on the concentration of the empirical covariance of isotropic sub-Gaussian random vectors.
We say a zero-mean random vector $f \in \C^d$ (this covers the real case as well) is isotropic and $K$--sub-Gaussian (for some $K > 0$) if, for all unit-norm $x \in \C^d$, $\ip{f}{x}$ is unit-variance (i.e., $\E\abs{\ip{f}{x}}^2 = 1$) and $K$--sub-Gaussian, that is, $\E e^{\abs{\ip{f}{x}}^2/K^2} \leq 2$.
Note that this is, in particular, true under the conditions of \Cref{thm:finite_subG}.
\begin{lemma}[e.g., {\cite[Thm.~1]{Koltchinskii2017}}]
	\label{lem:subG_conc}
	Let $f_1, \dots, f_n$ be independent zero-mean isotropic $K$--sub-Gaussian vectors
	in $\C^d$.
	There exists a constant $c > 0$ depending only on $K$ such that, for all $t \geq 1$, with probability at least $1 - e^{-t}$,
	\[
		\opnorm*{\frac{1}{n} \sum_{i=1}^n f_i f_i^* - I_d }
		\leq c \parens*{\sqrt{\frac{d}{n}} + \frac{d}{n} + \sqrt{\frac{t}{n}} + \frac{t}{n} }.
	\]
\end{lemma}
In particular, choosing $t = 2 \log n$,
for constants $c_1, c_2 > 0$ depending only on $K$,
if $n \geq c_1 d$, with probability at least $1 - n^{-2}$,
\begin{align*}
	\opnorm*{\frac{1}{n} \sum_{i=1}^n \real(f_i f_i^*) - I_d }
	&\leq \opnorm*{\frac{1}{n} \sum_{i=1}^n f_i f_i^* - I_d } \\
	&\leq c_2 \sqrt{\frac{d + \log n}{n}}.
\end{align*}

Next, the following result will be essential:
\begin{lemma}[{\cite{Krahmer2020}, as stated in \cite[Lem.~4]{McRae2026a}}]
	\label{lem:subG_lb}
	Under the conditions of \Cref{thm:finite_subG},
	there exist constants $c_1, c_2, c_3, c_4 > 0$ depending only on the properties of $\crdvar$ such that,
	for $n \geq c_1 d$,
	with probability at least $1 - c_2 e^{-c_3 n}$, for all $Z \succeq 0$,
	\[
		\frac{1}{n} \sum_{i=1}^n \abs{ \ip{A_i}{Z - Z_*} }
		\geq c_4 \nucnorm{Z - x_* x_*^*}.
	\]
\end{lemma}

The following two technical lemmas allow us to use the previous lemma in our framework:
\begin{lemma}
	\label{lem:ab_ineq}
	For all $X_1 \in \F^{d \times r_1}$, $X_2 \in \F^{d \times r_2}$,
	\[
	\frac{1}{n} \norm{\alpha(X_1) - \alpha(X_2)}^2 \geq \frac{\norm{\beta(X_1) - \beta(X_2)}_1^2}{n \norm{\alpha(X_1) + \alpha(X_2)}^2}.
	\]
\end{lemma}
\begin{proof}
	Convexity of the function $(x, y) \mapsto \frac{x^2}{y}$ and Jensen's inequality (Cauchy-Schwartz would also work) imply
	\begin{align*}
		&\frac{1}{n} \norm{\alpha(X_1) - \alpha(X_2)}^2 \\
		&\quad= \frac{1}{n} \sum_{i=1}^n (\ip{A_i}{X_1 X_1^*}^{1/2} - \ip{A_i}{X_2 X_2^*}^{1/2} )^2 \\
		&\quad= \frac{1}{n} \sum_{i=1}^n \frac{\abs{\ip{A_i}{X_1 X_1^*} - \ip{A_i}{X_2 X_2^*} }^2}{(\ip{A_i}{X_1 X_1^*}^{1/2} + \ip{A_i}{X_2 X_2^*}^{1/2} )^2} \\
		&\quad\geq \frac{ \parens*{ \frac{1}{n} \sum_{i=1}^n \abs{\ip{A_i}{X_1 X_1^*} - \ip{A_i}{X_2 X_2^*} } }^2 }{\frac{1}{n} \sum_{i=1}^n (\ip{A_i}{X_1 X_1^*}^{1/2} + \ip{A_i}{X_2 X_2^*}^{1/2} )^2} \\
		&\quad= \frac{\norm{\beta(X_1) - \beta(X_2)}_1^2}{n \norm{\alpha(X_1) + \alpha(X_2)}^2}.
	\end{align*}
\end{proof}

\begin{lemma}
	\label{lem:nucnorm_lb}
	Let $x_* \in \F^d$ and $X \in \F^{d \times p}$.
	There is a unit-norm $v \in \F^p$ such that
	\[
		\nucnorm{X X^* - x_* x_*}
		\geq \normF{X X^* - x_* x_*}
		\geq \frac{1}{2 \sqrt{2}} \norm{x_*} \norm{x_* - X v}.
	\]
	In particular, $(X v) v^*$ can be any best rank-1 approximation to $X$ in $\norm{\cdot}$.
\end{lemma}
\begin{proof}
	Let $v \in \F^p$ be a unit-norm leading right singular vector of $X$,
	which is equivalent to $(X v) v^*$ being a rank-1 projection of $X$ in $\norm{\cdot}$.
	We can furthermore choose $v$ (multiplying by a unit-magnitude number if needed) such that $\ip{Xv}{x_*} \geq 0$.
	As $X v v^* X^*$ is an optimal rank-1 approximation to $X X^*$,
	\begin{align*}
		0 &\leq \normF{X X^* - x_* x_*^*}^2 - \normF{X X^* - X v v^* X^*}^2 \\
		&= - \normF{X v v^* X^* - x_* x_*^*}^2 \\
		&\quad + 2 \ip{X v v^* X^* - x_* x_*^*}{X X^* - x_* x_*^*} \\
		&\leq - \normF{X v v^* X^* - x_* x_*^*}^2 \\
		&\quad + 2 \normF{X v v^* X^* - x_* x_*^*} \normF{X X^* - x_* x_*^*}.
	\end{align*}
	Hence
	\[
		\normF{X v v^* X^* - x_* x_*^*} \leq 2 \normF{X X^* - x_* x_*^*}.
	\]
	Furthermore,
	\begin{align*}
		\normF{X v v^* X^* - x_* x_*^*}^2
		&= \norm{Xv}^4 + \norm{x_*}^4 - 2 \ip{Xv}{x_*}^2 \\
		&\geq \frac{(\norm{Xv}^2 + \norm{x_*}^2)^2}{2} - 2 \ip{Xv}{x_*}^2 \\
		&= \frac{1}{2} \norm{Xv + x_*}^2 \norm{Xv - x_*}^2 \\
		&\geq \frac{1}{2} \norm{x_*}^2 \norm{x_* - Xv}^2.
	\end{align*}
	Combining this with the previous bound completes the proof.
\end{proof}

Next, the following bound will prove useful:
\begin{lemma}
	\label{lem:a_ineq}
	Consider the model \eqref{eq:gen_model_sqrt}.
	For any $\lambda \geq 0$,
	any $X$ that is a first-order critical point of \eqref{eq:opt_ncvx_sqrt},
	or $X = \recmap U$ for any feasible $U$ of \eqref{eq:phasecut_gen},
	satisfies
	\[
		\frac{1}{n} \norm{\alpha(X)}^2 + \lambda \norm{X}^2 \leq \frac{1}{n} \norm{y}^2 \leq \frac{1}{n} (\norm{\alpha(X_*)} + \norm{\varepsilon})^2.
	\]
\end{lemma}
\begin{proof}
	First, we consider a first-order critical point $X$ of \eqref{eq:opt_ncvx_sqrt} (see \Cref{lem:socp_char}).
	The condition $\nabla L_\lambda(X X^*) X = 0$ implies
	\begin{align*}
		0 &= \ip{ \nabla L_\lambda(X X^*) }{X X^*} \\
		&= \frac{1}{n} \sum_{i=1}^n ( \ip{A_i}{X X^*} - y_i \ip{A_i}{X X^*}^{1/2} ) + \lambda \norm{X}^2 \\
		&= \frac{1}{n} \norm{\alpha(X)}^2 - \frac{1}{n} \ip{y}{\alpha(X)} + \lambda \norm{X}^2.
	\end{align*}
	This implies
	\begin{align*}
		\frac{1}{n} \norm{\alpha(X)}^2 + \lambda \norm{X}^2
		&\leq \frac{1}{n} \norm{y} \norm{\alpha(X)} \\
		&\leq \frac{1}{\sqrt{n}} \norm{y} \sqrt{\frac{1}{n} \norm{\alpha(X)}^2 + \lambda \norm{X}^2 },
	\end{align*}
	from which the claimed bound follows.
	
	Second, we consider $X = \recmap U$, where $U$ is feasible for \eqref{eq:phasecut_gen}.
	The identity \eqref{eq:quad_id} implies
	\begin{align*}
		\frac{1}{n} \norm{\alpha(X)}^2 + \lambda \norm{X}^2
		&= \frac{1}{n} \norm{F X}^2 + \lambda \norm{X}^2 \\
		&= \frac{1}{n} \norm{\diag(y) U}^2 - \ip{M_\lambda U}{U} \\
		&\leq \frac{1}{n} \norm{y}^2,
	\end{align*}
	where the inequality is due to $M_\lambda \succeq 0$.
\end{proof}

With these, we can prove the main finite-dimensional statistical result:
\begin{proof}[Proof of \Cref{thm:finite_subG}]
	Throughout the proof, we use $c$, $c'$, etc.\ to denote positive constants which may change from one usage to another but which do not depend on the problem parameters $n$ and $d$ (though, as discussed in the theorem statement, they may depend on properties of $\crdvar$).
	
	Assuming $n \geq c d$, we can use \Cref{lem:subG_conc,lem:subG_lb} (combining and simplifying the failure probabilities) to obtain, with probability at least $1 - c n^{-2}$, for all $r' \geq 1$,
	\begin{gather}
		\label{eq:alphanorm_conc}
		c \norm{X'}^2 \leq \frac{1}{n} \norm{\alpha(X')}^2 \leq c' \norm{X'}^2 \quad \forall X' \in \F^{d \times r'},
	\end{gather}
	and
	\begin{gather}
		\label{eq:betal1_ineq}
		\frac{1}{n} \norm{\beta(X) - \beta(x_*)}_1
		\geq c \nucnorm{X X^* - x_* x_*^*} \quad \forall X \in \F^{d \times p}.
	\end{gather}
	From now on, we assume this event holds.
	
	We divide the rest of the proof into two cases depending on the norm of $\varepsilon$.
	First, in the (trivial) case that $\norm{\varepsilon}^2 \geq n \norm{x_*}^2$,
	for $X$ as in the theorem statement, \Cref{lem:a_ineq} and \eqref{eq:alphanorm_conc} give
	\begin{align*}
		\nucnorm{X X^* - x_* x_*^*}
		&\leq \norm{X}^2 + \norm{x_*}^2 \\
		&\leq \frac{c}{n}\norm{\alpha(X)}^2 + \norm{x_*}^2  \\
		&\leq \frac{c}{n}( \norm{\varepsilon}^2 + \norm{\alpha(x_*)}^2) + \norm{x_*}^2 \\
		&\leq c \parens*{ \frac{\norm{\varepsilon}^2}{n} + \norm{x_*}^2 } \\
		&\leq c\frac{\norm{\varepsilon}^2}{n},
	\end{align*}
	and, for $\xhat$ as in the theorem statement,
	\begin{align*}
		\norm{\xhat - x_*}
		&\leq \norm{\xhat} + \norm{x_*} \\
		&\leq \norm{X} + \norm{x_*} \\
		&\leq c \frac{\norm{\varepsilon}}{\sqrt{n}}.
	\end{align*}
	Thus, from now on, assume $\norm{\varepsilon}^2 \leq n \norm{x_*}^2 \leq c \norm{\alpha(x_*)}^2$.
	Together with (again) \Cref{lem:a_ineq} and \eqref{eq:alphanorm_conc},
	we obtain
	\begin{align*}
		\norm{\alpha(X) + \alpha(x_*)}
		&\leq \norm{\alpha(X)} + \norm{\alpha(x_*)} \\
		&\leq \norm{\varepsilon} + 2 \norm{\alpha(x_*)} \\
		&\leq c \sqrt{n} \norm{x_*}.
	\end{align*}
	Together with this last bound,
	\Cref{lem:ab_ineq} and \eqref{eq:betal1_ineq} give
	\begin{align*}
		\frac{1}{n} \norm{\alpha(X) - \alpha(x_*)}^2
		&\geq \frac{\norm{\beta(X) - \beta(x_*)}_1^2}{n \norm{\alpha(X) + \alpha(x_*)}^2} \\
		&\geq \frac{c}{\norm{x_*}^2} \nucnorm{X X^* - x_* x_*^*}^2.
	\end{align*}
	By \Cref{lem:nucnorm_lb}, there is $v \in \F^p$ such that
	\[
		\nucnorm{X X^* - x_* x_*^*}^2 \geq c \norm{x_*}^2 \norm{x_* - Xv}^2.
	\]
	We therefore have
	\begin{equation}
		\label{eq:alpha_err_lb}
		\frac{1}{n} \norm{\alpha(X) - \alpha(x_*)}^2
		\geq \frac{c}{\norm{x_*}^2} \nucnorm{X X^* - x_* x_*^*}^2 + c  \norm{x_* - Xv}^2.
	\end{equation}
	To finish, we first consider the problem \eqref{eq:opt_ncvx_sqrt}
	and then make suitable modifications for the PhaseCut approach of \eqref{eq:phasecut_gen}.
	If $X$ is a second-order critical point of \eqref{eq:opt_ncvx_sqrt},
	\Cref{thm:landscape_sqrt} gives
	\begin{align*}
		&\frac{1}{n} \norm{\alpha(X) - \alpha(x_*)}^2 \\
		&\quad\leq \frac{2}{n} \ip{\varepsilon}{\alpha(X) - \alpha(x_*)} + \frac{1}{\constF p - 1} \cdot \frac{1}{n} \norm{\alpha(x_* - Xv)}^2 \\
		&\quad\leq \frac{2}{n} \norm{\varepsilon} \norm{\alpha(X) - \alpha(x_*)} + \frac{c}{pn} \norm{\alpha(x_* - Xv)}^2.
	\end{align*}
	Some algebra together with \eqref{eq:alphanorm_conc} then implies
	\begin{equation}
		\label{eq:inter_ineq_sqrt}
		\begin{aligned}
			\frac{1}{n} \norm{\alpha(X) - \alpha(x_*)}^2
			&\leq c \frac{\norm{\varepsilon}^2}{n} + \frac{c'}{pn } \norm{\alpha(x_* - Xv)}^2 \\
			&\leq c \frac{\norm{\varepsilon}^2}{n} + \frac{c'}{p} \norm{x_* - Xv}^2.
		\end{aligned}
	\end{equation}
	If $p \geq c$,
	combining \eqref{eq:alpha_err_lb} with \eqref{eq:inter_ineq_sqrt} gives
	\[
		\frac{\nucnorm{X X^* - x_* x_*^*}^2}{\norm{x_*}^2}  \leq c \frac{\norm{\varepsilon}^2}{n},
	\]
	which implies the claimed bound on $\nucnorm{X X^* - x_* x_*^*}$.
	The bound on $\min_{\abs{s} = 1}~\norm{\xhat - s x_*}$ also follows by \Cref{lem:nucnorm_lb}.
	
	Now, suppose $X = \recmap U$, where $U$ is a second-order critical point of \eqref{eq:phasecut_gen}.
	\Cref{thm:landscape_phasecut} and some algebra give, similarly to \eqref{eq:inter_ineq_sqrt},
	\begin{align*}
%		\label{eq:inter_ineq_pc}
		\frac{1}{n} \norm{\alpha(X) - \alpha(x_*)}^2
		&\leq c \frac{\norm{\varepsilon}^2}{n} + \frac{c'}{p n} \norm{F(X_\lambda - Xv)}^2 \\
		&\leq c \frac{\norm{\varepsilon}^2}{n} + \frac{c'}{p n} \norm{F(x_* - Xv)}^2 \\
		&\leq c \frac{\norm{\varepsilon}^2}{n} + \frac{c'}{p} \norm{x_* - Xv}^2,
	\end{align*}
	where the second inequality uses \eqref{eq:Xlam_errbd} (with $\lambda = 0$),
	and $v$ is chosen to be the same as in the previous case.
	This is identical (within constants) to \eqref{eq:inter_ineq_sqrt};
	the result follows similarly.
\end{proof}

We now turn to the proof of the infinite-dimensional result \Cref{thm:infdim}.
\Cref{thm:landscape_sqrt,thm:landscape_phasecut} and the intermediate technical results \Cref{lem:ab_ineq,lem:a_ineq} adapt in obvious ways with $\F^d$ replaced by $\scrH$.
We will also use \Cref{lem:subG_lb,lem:subG_conc,lem:nucnorm_lb} again but only on a finite-dimensional subspace of $\scrH$, so no adaptation is necessary.

In addition, the following infinite-dimensional Gaussian counterpart to \Cref{lem:subG_conc} will be useful:
\begin{lemma}[{\cite[Cor.~2]{Koltchinskii2017}}]
	\label{lem:infdim_gauss_conc}
	Let $z_1, \dots, z_n$ be i.i.d.\ Gaussian random vectors with covariance $\Sigma$ in a Hilbert space $\scrH$.
	Let
	\[
		r(\Sigma) \coloneqq \frac{\tr \Sigma}{\opnorm{\Sigma}}.
	\]
	Let
	\[
		\Sigmahat \coloneqq \frac{1}{n} \sum_{i=1}^n z_i^{\otimes 2}
	\]
	be the sample covariance.
	There is a universal constant $c > 0$ such that,
	for any $t \geq 1$, with probability at least $1 - e^{-t}$,
	\[
		\opnorm{\Sigmahat - \Sigma}
		\leq c \opnorm{\Sigma}\parens*{ \sqrt{\frac{r(\Sigma)}{n}} + \frac{r(\Sigma)}{n} + \sqrt{\frac{t}{n}} + \frac{t}{n} }.
	\]	
\end{lemma}
In particular, choosing $t = 2\log n$, this implies, with probability at least $1 - n^{-2}$, for another universal constant $c' > 0$,
\[
	\opnorm{\Sigmahat} \leq c' \parens*{\opnorm{\Sigma} + \frac{\tr \Sigma}{n}}.
\]

With this and the tools already presented for the proof of \Cref{thm:finite_subG}, we can proceed to the proof of our infinite-dimensional result:
\newcommand{\dimG}{d_G}
\newcommand{\SigG}{\Sigma_G}
\newcommand{\SigGp}{\Sigma_{G^\perp}}
\begin{proof}[Proof of \Cref{thm:infdim}]
	Let $u_1, \dots, u_d \in \scrH$ be the eigenvectors of $\Sigma$ corresponding to eigenvalues $\sigma_1 \geq \cdots \geq \sigma_d$.
	We set
	\[
		G \coloneqq \spn\{ x_*, u_1, \dots, u_d \}.
	\]
	Note that $\dimG \coloneqq \dim(G)$ is either $d$ or $d+1$.
	Let $G^\perp$ be the orthogonal complement of $G$ in $\scrH$.
	Denote by $\PG$ and $\PGp$ the orthogonal projection operators onto (respectively) $G$ and $G^\perp$.
	
	Note that
	$\SigG \coloneqq \PG \Sigma \PG$ has rank at most $\dimG$, and $\PG f \in \range(\SigG)$ almost surely.
	We then have that $\{\Sigma_G^{-1/2} \PG f_i\}_{i=1}^n$ are i.i.d.\ standard Gaussian (standard complex Gaussian if $\F = \C$---recall that we required $f$ to be circularly symmetric in this case) vectors on $\range(\SigG)$ with respect to any orthonormal (in $L_2$) basis.
	Assuming $n \geq c \dimG$, \Cref{lem:subG_conc,lem:subG_lb} imply that
	\begin{equation*}
		c \SigG \preceq \frac{1}{n} \sum_{i=1}^n (\PG f_i)^{\otimes 2} \preceq c' \SigG %\label{eq:PG_conc},
	\end{equation*}
	and, for all $X \in \scrH^p$,
	\begin{equation}
		\label{eq:PG_lb}
		\frac{1}{n} \norm{\beta( \PG X) - \beta(x_*)}_1
		\geq c \nucnormlt{\PG X X^* \PG - x_* x_*^*}.
	\end{equation}
	Furthermore, \Cref{lem:infdim_gauss_conc} (with covariance $\SigGp$; see also the discussion following that lemma) implies
	\begin{equation}
		\label{eq:PGp_opnorm_delta}
		\opnorm*{\frac{1}{n} \sum_{i=1}^n (\PGp f_i)^{\otimes 2}} \leq c \parens*{\opnorm{\SigGp} + \frac{\tr \SigGp}{n}} \eqqcolon \delta.
	\end{equation}
	Combining the failure probabilities of these events with a union bound,
	these inequalities hold with probability at least $1 - c n^{-2}$.
	From now on, assume these hold.
	In particular,
	for all $X' \in \scrH^{r'}$ (for any $r' \geq 1$),
	we have
	\begin{equation}
		c \normlt{\PG X'}^2 \leq \frac{1}{n} \norm{\alpha(\PG X')}^2 \leq c' \normlt{\PG X'}^2 \label{eq:PG_conc},
	\end{equation}
	and
	\begin{equation}
		\frac{1}{n}\norm{\alpha(\PGp X')}^2 \leq \delta \norm{\PGp(X')}^2
		\leq \delta \norm{X'}^2.
		\label{eq:PGp_ub}
	\end{equation}
	
	Noting that $\opnorm{\SigGp} \leq \sigma_{\tgtdim+1}$ and $\tr \SigGp \leq \sum_{m > \tgtdim} \sigma_m$,
	\eqref{eq:PGp_opnorm_delta} allows us to have $\lambda \geq c \delta$.
	We can then derive the following simple inequality for any $X$ satisfying the conditions of the theorem (we defer the proof to later):
	\begin{equation}
		\label{eq:Xnorm_ub}
		\normlt{X}^2 \leq c \parens*{ \normlt{x_*}^2 + \frac{\norm{\varepsilon}^2}{n} }.
	\end{equation}
	We first consider the (trivial) case where $\norm{\varepsilon}^2 \geq n \norm{x_*}^2$.
	From \eqref{eq:Xnorm_ub} and the triangle inequality, we have
	\begin{align*}
		\nucnormlt{X X^* - x_* x_*^*}
		&\leq \normlt{X}^2 + \normlt{x_*}^2 \\
		&\leq c \frac{\norm{\varepsilon}^2}{n},
	\end{align*}
	and
	\begin{align*}
		\normlt{\xhat - x_*}
		&\leq \normlt{X} + \normlt{x_*} \\
		&\leq c \frac{\norm{\varepsilon}}{\sqrt{n}}.
	\end{align*}
	Thus from now on assume $\norm{\varepsilon}^2 \leq n \normlt{x_*}^2$.
	
	By \eqref{eq:Xnorm_ub} and related calculations (again, we defer these to later),
	we can then derive a number of useful inequalities:	
	\begin{align}
		\normlt{X}&\leq c \normlt{x_*} \label{eq:X_ub_2} \\
		\normlt{\PG X} &\leq c \normlt{x_*} \label{eq:PGX_ub} \\
%		\nucnormlt{X X^* - x_* x_*} &\leq c \normlt{x_*}^2 \\
%		\nucnormlt{\PG X X^* \PG - x_* x_*^*} &\leq c \normlt{x_*}^2 \\
		\norm{\alpha(\PG X) + \alpha(x_*) } &\leq c \sqrt{n} \normlt{x_*}. \label{eq:alphaPGX_ub}
	\end{align}
	\Cref{lem:ab_ineq}, \eqref{eq:PG_lb}, \eqref{eq:PGp_ub}, \eqref{eq:alphaPGX_ub}, and the triangle inequality imply
	\begin{align*}
		&\frac{1}{\sqrt{n}} \norm{\alpha(X) - \alpha(x_*)} \\
		&\quad\geq \frac{1}{\sqrt{n}} \norm{\alpha(\PG X) - \alpha(x_*)} - \frac{1}{\sqrt{n}} \norm{\alpha(X) - \alpha(\PG X)} \\
		&\quad\geq \frac{\norm{\beta(\PG X) - \beta(x_*)}_1}{\sqrt{n} \norm{\alpha(\PG X) + \alpha(x_*)}} - \frac{1}{\sqrt{n}} \norm{\alpha(\PGp X)} \\
		&\quad\geq c \frac{\nucnormlt{\PG X X^* \PG - x_* x_*^*}}{\normlt{x_*}} - \delta^{1/2} \norm{X}.
	\end{align*}
	Furthermore, using \eqref{eq:X_ub_2}, \eqref{eq:PGX_ub}, we have
	\begin{align*}
		&\nucnormlt{\PG X X^* \PG - x_* x_*^*} \\
		&\quad\geq \nucnormlt{X X^* - x_* x_*^*} \\
		&\qquad - \nucnormlt{\PGp X X^*} - \nucnormlt{\PG X X^* \PGp} \\
		&\quad\geq \nucnormlt{X X^* - x_* x_*^*} \\
		&\qquad - \normlt{\PGp X} \normlt{X} - \normlt{\PG X} \normlt{\PGp X} \\
		&\quad\geq \nucnormlt{X X^* - x_* x_*^*} - c \normlt{x_*} \opnorm{\SigGp}^{1/2} \norm{X}.
	\end{align*}
	Next, \Cref{lem:nucnorm_lb} implies that, for some unit-norm $v \in \F^p$,
	\[
		\nucnormlt{\PG X X^* \PG - x_* x_*^*}
		\geq c \normlt{x_*} \normlt{x_* - \PG X v}.
	\]
	Combining the previous three displays and the inequality $\delta \geq c \opnorm{\SigGp}$ from \eqref{eq:PGp_opnorm_delta},
	we have
	\begin{align*}
		&\frac{1}{\sqrt{n}} \norm{\alpha(X) - \alpha(x_*)} \\
		&\quad\geq \frac{c}{\normlt{x_*}} \nucnormlt{X X^* - x_* x_*^*} \\
		&\qquad+ c'\normlt{x_* - \PG X v} - c'' \delta^{1/2} \norm{X},
	\end{align*}
	which implies
	\begin{equation}
		\label{eq:infdim_alpha_lb}
		\begin{aligned}
		&\frac{1}{n} \norm{\alpha(X) - \alpha(x_*)}^2 \\
		&\quad \geq \frac{c}{\normlt{x_*}^2} \nucnormlt{X X^* - x_* x_*^*}^2 \\
		&\qquad	+ c' \normlt{x_* - \PG X v}^2 - c'' \delta \norm{X}^2.
		\end{aligned}
	\end{equation}
	We now turn to upper bounding $\norm{\alpha(X) - \alpha(x_*)}^2$.
	We first consider the case where $X$ is a second-order critical point of \eqref{eq:opt_ncvx_sqrt}.
	\Cref{thm:landscape_sqrt} implies, for the same (unit-norm) $v$ as above,
	\begin{align*}
		&\frac{1}{n} \norm{\alpha(X) - \alpha(x_*)}^2 \\
		&\quad \leq \frac{2}{n} \ip{\varepsilon}{\alpha(X) - \alpha(x_*)} + \lambda (\norm{x_*}^2 - \norm{X}^2) \\
		&\qquad + \frac{1}{\constF p - 1} \parens*{\frac{1}{n} \norm{\alpha(x_* - X v)}^2 + \lambda \norm{x_* - X v}^2} \\
		&\quad\leq \frac{2}{n} \norm{\varepsilon} \norm{\alpha(X) - \alpha(x_*)} + \lambda (\norm{x_*}^2 - \norm{X}^2) \\
		&\qquad + \frac{c}{p} \parens*{ \normlt{x_* - \PG X v }^2 + \delta \norm{Xv}^2 + \lambda (\norm{x_*}^2 + \norm{Xv}^2) },
	\end{align*}
	where the second inequality used \eqref{eq:PG_conc} and \eqref{eq:PGp_ub}.
	This implies, recalling $\lambda \geq c \delta$,
	\begin{equation}
		\label{eq:infdim_alpha_ub}
		\begin{aligned}
		&\frac{1}{n} \norm{\alpha(X) - \alpha(x_*)}^2 \\
		&\quad \leq c \frac{\norm{\varepsilon}^2}{n} + c' \lambda (\norm{x_*}^2 - \norm{X}^2) \\
		&\qquad + \frac{c''}{p} ( \normlt{x_* - \PG X v }^2 + \lambda (\norm{x_*}^2 + \norm{X}^2 ) ).
		\end{aligned}
	\end{equation}
	Together with \eqref{eq:infdim_alpha_lb},
	with $p \geq c$ and $\lambda \geq c \delta$,
	we obtain
	\begin{equation*}
		\frac{1}{\normlt{x_*}^2} \nucnormlt{X X^* - x_* x_*^*}^2
		\leq c \frac{\norm{\varepsilon}^2}{n} + c' \lambda \norm{x_*}^2.
	\end{equation*}
	The result immediately follows (again \Cref{lem:nucnorm_lb} to bound $\min_{\abs{s} = 1}~\normlt{\xhat - s x_*}$).
	
	In the case where $X = \recmap U$ for a second-order critical point $U$ of \eqref{eq:phasecut_gen},
	similarly to the proof of \Cref{thm:finite_subG} above,
	\Cref{thm:landscape_phasecut} and similar arguments as above give an inequality identical (within constants) to \eqref{eq:infdim_alpha_ub}.
	The result follows similarly.

	Finally, we prove the inequalities \eqref{eq:Xnorm_ub}--\eqref{eq:alphaPGX_ub}.
	By \eqref{eq:PG_conc} and \eqref{eq:PGp_ub},
	\begin{align*}
		\frac{1}{\sqrt{n}}\norm{\alpha(X)}
		&\geq \frac{1}{\sqrt{n}} \norm{\alpha(\PG X)} - \frac{1}{\sqrt{n}} \norm{\alpha(\PGp X)} \\
		&\geq c \normlt{\PG X} - \sqrt{\delta} \norm{X} \\
		&\geq c \normlt{X} - c' \sqrt{\delta} \norm{X},
	\end{align*}
	where the last inequality uses the fact that $\delta \geq c \opnorm{\SigG}$.
	Some algebra implies
	\[
		\frac{1}{n} \norm{\alpha(X)}^2 \geq c \normlt{X}^2 - c' \delta \norm{X}^2
	\]
	and
	\[
		\frac{1}{n} \norm{\alpha(X)}^2 \geq c \normlt{\PG X}^2 - c' \delta \norm{X}^2.
	\]
	Together with \Cref{lem:a_ineq} and $\lambda \geq c \delta$,
	this gives, for $X$ as in the theorem statement,
	\begin{align*}
		\max\{ \normlt{X}^2, \normlt{\PG X}^2 \}
		&\leq \frac{c}{n} \norm{\alpha(X)}^2 + c' \delta \norm{X}^2 \\
		&\leq c \parens*{ \frac{1}{n} \norm{\alpha(X)}^2 + \lambda \norm{X}^2 } \\
		&\leq \frac{c}{n} \parens{ \norm{\alpha(x_*)}^2 + \norm{\varepsilon}^2 } \\
		&\leq c \normlt{x_*}^2 + c' \frac{\norm{\varepsilon}^2}{n},
	\end{align*}
	where the last inequality again used \eqref{eq:PG_conc}.
	We thus have \eqref{eq:Xnorm_ub}.
	In the case $\norm{\varepsilon}^2 \leq n \normlt{x_*}^2$, \eqref{eq:X_ub_2} and \eqref{eq:PGX_ub} also follow immediately.
	Combining \eqref{eq:PGX_ub} with \eqref{eq:PG_conc} gives \eqref{eq:alphaPGX_ub}.
\end{proof}

\section*{Acknowledgments}
\ifMS \else \fundingack{} \fi
The author thanks Jonathan Dong, Irène Waldspurger, and Christopher Criscitiello for helpful inspiration, discussions, and suggestions.
The author also thanks the anonymous reviewers for their many helpful suggestions to improve the paper and for pointing out many minor errors.

%\ifMS
\bibliographystyle{IEEEtran}
\bibliography{refs}

\ifMS
\begin{IEEEbiographynophoto}{Andrew D. McRae}
	received B.S., M.S., and Ph.D.\ degrees from the Georgia Institute of Technology in 2015, 2016, and 2022, respectively.
	He is currently an assistant professor at the École nationale des ponts et chaussées (ENPC) in the CERMICS group.
	Prior to this, in 2022-2025, he was a postdoctoral researcher
	with the Institute of Mathematics at the École Polytechnique Fédérale de Lausanne (EPFL).
	His research interests are in high-dimensional statistics and signal processing with a particular focus on the role of	optimization.
\end{IEEEbiographynophoto}
\fi
\end{document}